\documentclass[leqno,b5paper]{amsart}

\usepackage{amsmath,amsfonts,amsthm,latexsym}
\usepackage{mathrsfs, textcomp}
\newtheorem{theorem}{Theorem}[section]
\newtheorem{lemma}[theorem]{Lemma}
\newtheorem{proposition}[theorem]{Proposition}
\newtheorem{corollary}[theorem]{Corollary}

\theoremstyle{definition}
\newtheorem{definition}[theorem]{Definition}
\newtheorem{example}[theorem]{Example}
\newtheorem{obs}[theorem]{Observation}
\newtheorem{remark}[theorem]{Remark}

\numberwithin{equation}{section}

\usepackage{amssymb}
\usepackage{amscd}
\usepackage{eqlist, color}
\usepackage{mathrsfs, textcomp}
\usepackage{color, eqlist}
\usepackage{graphicx}
\usepackage{bm}
\usepackage{amsmath,amsfonts,amsthm,latexsym}
\usepackage[all]{xy}

\DeclareMathOperator{\Nd}{Nd}

\DeclareMathOperator{\NdS}{\mathfrak {ND}}

\DeclareMathOperator{\inte}{int}

\DeclareMathOperator{\Int}{int}
\DeclareMathOperator{\bd}{bd}
\DeclareMathOperator{\Inac}{Inac}
\DeclareMathOperator{\Ainac}{Ainac}
\DeclareMathOperator{\two}{\textbf{2}}
\DeclareMathOperator{\three}{\textbf{3}}
\DeclareMathOperator{\cl}{cl}

\DeclareMathOperator{\Rs}{Rs}

\theoremstyle{definition}
\theoremstyle{definition}
\theoremstyle{definition}
\theoremstyle{definition}
\theoremstyle{definition}
\theoremstyle{remark}
\theoremstyle{definition}

\theoremstyle{definition}

\begin{document}
	\title[Maximal Nowhere Dense Sublocales]
	{On maximal nowhere dense sublocales}

	\author{Mbekezeli Nxumalo}
	\address{Department of Mathematical Sciences, University of South Africa, P.O. Box 392, 0003 Pretoria, SOUTH AFRICA.}
	\address{Department of Mathematics, Rhodes University, P.O. Box 94, Grahamstown 6140, South Africa.}
	\email{sibahlezwide@gmail.com}
	\thanks{This paper is part of a Ph.D thesis written under the supervision of Prof. Themba Dube. The author  acknowledges and highly appreciates the support he received from his supervisor, and also acknowledges funding from the National Research Foundation of South Africa under Grant 134159}
	\dedicatory{}
	
	
	\let\thefootnote\relax\footnote{}
	
	\begin{abstract} The aim of this paper is to study some variants of nowhere dense sublocales called maximal nowhere dense and homogeneous maximal nowhere dense sublocales. These concepts were initially introduced by Veksler in classical topology. We give some general properties of these sublocales and further examine their relationship with both inaccessible sublocales and remote sublocales. It turns out that a locale has all of its non-void nowhere dense sublocales maximal nowhere dense precisely when all of its its non-void nowhere dense sublocales are inaccessible. We show that the Booleanization of a locale is inaccessible with respect to every dense and open sublocale. In connection to remote sublocales, we prove that, if the supplement of an open dense sublocale $S$ is homogeneous maximal nowhere dense, then every $S^{\#}$-remote sublocale is $^{*}$-remote from $S$. Every open localic map that sends dense elements to dense elements preserves and reflects maximal nowhere dense sublocales. If such a localic map is further injective, then it sends homogeneous maximal nowhere dense sublocales back and forth.
	\end{abstract}
	
	\maketitle
	
	
	\section*{Introduction}\label{sect0}
	
	Veksler \cite{V}, in 1975, introduced the concepts of a maximal nowhere dense subset and a homogeneous maximal nowhere dense subset of a Tychonoff space. He defined a closed subset of a Tychonoff space to be \textit{maximal nowhere dense} if it is not a nowhere dense subset of any other closed nowhere dense subset of the space, and \textit{homogeneous maximal nowhere dense} provided that each of its non-void regular-closed subsets is maximal nowhere dense in the space. These subsets have since appeared in several articles such as \cite{K} and \cite{V1}. In \cite{K}, Koldunov considered connections of these subsets with $\theta$-sets, $P$-sets and $P'$-sets, and in \cite{V1}, Veksler applied maximal nowhere dense subsets to problems of the existence of remote points and of weak $P$-points. 
	
	Nowhere dense sublocales were initially introduced by Plewe in \cite{P} as those sublocales missing the Booleanization of a locale. In locale theory, these sublocales appear in a number of papers, for instance in \cite{D2}, where the author used them to characterize submaximal locales. In \cite{N}, we used nowhere dense sublocales to define \textit{remote sublocales} which are those sublocales missing all nowhere dense sublocales. We further verified that the Booleanization of a locale is a largest remote sublocale. Knowing how nowhere dense sublocales relate with the Booleanization of a locale, it is of interest to know how some types of nowhere dense sublocales also relate with the Booleanization and to also know how these types of nowhere dense sublocales fit in the point-free context independently to their already documented use in classical topology. In this paper, we introduce and study two variants of nowhere density called maximal nowhere dense and homogeneous maximal nowhere dense sublocales. These notions are transferred from classical topology. Here though, we generalize the scope to arbitrary sublocales instead of closed sublocales.  We provide general results of these sublocales and show among other things that maximal nowhere dense sublocales are precisely those sublocales that are not nowhere in any nowhere dense sublocale. It turns out that  in the category of T$_{D}$-spaces, the localic definitions of maximal nowhere dense as well as homogeneous maximal nowhere sublocales are conservative in locales in the sense that a subset of a T$_{D}$-space is maximal nowhere dense (resp. homogeneous maximal nowhere dense) if and only the sublocale it induces is maximal nowhere dense (resp. homogeneous maximal nowhere dense) in the locale of opens. In a non-Boolean strongly submaximal locale, the supplement of the Booleanization is maximal nowhere dense. We study connnections between these sublocales and inaccessible as well as remote sublocales. It turns out that a locale has all of its non-void nowhere dense sublocales maximal nowhere dense precisely when all of its its non-void nowhere dense sublocales are inaccessible. Furthermore, the Booleanization of a locale is inaccessible with respect to a dense and open sublocale. In relation to remote sublocales, we show that if the supplement of an open dense sublocale $S$ is homogeneous maximal nowhere dense, then every $S^{\#}$-remote sublocale is $^{*}$-remote from $S$. 
	
	This paper is organized as follows. The first section introduces the basic background. In the second section, we define maximal nowhere dense sublocales. We characterize these sublocales where we show that a sublocale is maximal nowhere dense if and only if its closure is maximal nowhere dense. We use this property of maximal nowhere density to prove among other things that the localic definition of maximal nowhere density is conservative in locales. As far as maximality of objects is concerned, one would expect a sublocale to be maximal nowhere dense if it is not contained in any other nowhere dense sublocale different from itself. This is not the concept that we use in this paper, however we show that for a locale to have a nowhere dense sublocale that is not contained in any other nowhere dense sublocale, it is necessary and sufficient for the join of all of its nowhere dense sublocales to be nowhere dense.
	
	In the third section, we introduce the concept of a homogeneous maximal nowhere dense sublocale and show that the introduced localic definition is conservative in locales. We prove that homogeneous maximal nowhere density is regular-closed hereditary. 
	
	In section four,  we explore a relationship between maximal nowhere density and inaccessibility. Veksler \cite{V} introduced inaccessible and almost inaccessible points of a Tychonoff space. We transfer these concepts to completely regular locales and further define inaccessible and almost inaccessible sublocales of arbitrary locales. We use these sublocales to characterize locales in which all of their non-void nowhere dense sublocales are maximal nowhere dense.

	  The fifth section discusses a connection between maximal nowhere density and remoteness.

	The last section studies preservation and reflection of maximal nowhere density. We prove that every open localic map that sends dense elements to dense elements preserves and reflects maximal nowhere dense sublocales and if such a localic map is further injective, then it sends homogeneous maximal nowhere dense sublocales back and forth.
	
		\section{Preliminaries}\label{sect1}
		The book \cite{PP1} is our main reference for the theory of locales and sublocales. 
	\subsection{Locales }\label{subsect11}
	
	We recall that a \textit{locale} $L$ is a complete lattice satisfying the following infinite distributive law:
	
	$$a\wedge\bigvee B=\bigvee\{a\wedge b:b\in B\}$$ for every $a\in L$, $B\subseteq L$. 
	The top element and the bottom element of a locale $L$ are denoted by $1_{L}$ and $0_{L}$, respectively, with subscripts dropped if there is no possibility of confusion. An element $p\in L$ is called a \textit{point} if  $p<1$ and $a\wedge b\leq p$ implies $a\leq p$ or $b\leq p$, for all $a,b\in L$. The \textit{pseudocomplement} of an element $a\in L$ is the element $a^{\ast}=\bigvee\{x\in L:x\wedge a=0\}$. An element $a\in L$ is \textit{dense} if $a^{\ast}=0$, \textit{rather below} $b\in L$, denoted by $a\prec b$, if $a^{*}\vee b=1$, \textit{complemented} if $a\vee a^{\ast}=1$ and \textit{completely below} $b\in L$, denoted by $a\prec\!\!\prec b$, if there is a sequence $(x_{q})$ of elements of $L$ indexed by $\mathbb{Q}\cap[0,1]$ such that $a=0$, $b=1$ and $x_{q}\prec x_{r}$ whenever $q<r$. A locale $L$ is \textit{Boolean} if all of its elements are complemented and \textit{completely regular} in case  $a={\bigvee}\{x\in L:x\prec\!\!\prec a\}$ for all $a\in L$.

	A \textit{localic map} is an infimum-preserving function $f:L\rightarrow M$ between locales such that its left adjoint $f^{*}$, called the associated \textit{frame homomorphism}, preserves binary meets. 
	
	\subsection{Sublocales}\label{subsect12}
	A \textit{sublocale} of a locale $L$ is a subset $S\subseteq L$ such that (1) $\bigwedge A\in S$ for all $A\subseteq S$, and (2) for all $x\in L$ and $s\in S$, $x\rightarrow s\in S$, where $\rightarrow$ is a \textit{Heyting operation} on $L$ satisfying: $$a\leq b\rightarrow c\quad\text{if and only if}\quad a\wedge b\leq c$$ for every $a,b,c\in L$. We denote by $\mathcal{S}(L)$ the collection of all sublocales of a locale $L$. A sublocale $S$ of a locale $L$ is \textit{void} if $S=\mathsf{O}=\{1\}$,  \textit{non-void} provided that $S\neq\mathsf{O}$, and \textit{complemented} if it has a complement in $\mathcal{S}(L)$. It \textit{misses} a sublocale $T\subseteq L$ provided that $S\cap T=\mathsf{O}$. Every $S\in\mathcal{S}(L)$ has a \textit{supplement}, denoted by $L\smallsetminus S$ or $S^{\#}$, which is the smallest sublocale $T$ of $L$ such that $S\vee T=L$. A sublocale $S\subseteq L$ is \textit{linear} if  $$S\cap \bigvee\{C_{i}:i\in I\}=\bigvee\{S\cap C_{i}:i\in I\}$$ for each family $\{C_{i}:i\in I\}\subseteq \mathcal{S}(L)$. Complemented sublocales are linear. The sublocales
	$${\mathfrak{c}}(a)=\{x\in L:a\leq x\}\quad\text{and}\quad \mathfrak{o}(a)=\{a\rightarrow x:x\in L\},$$ of a locale $L$ are respectively the \textit{closed} and \textit{open} sublocales induced by $a\in L$. They are complements of each other. A sublocale is \textit{clopen} if it is both closed and open. For any $S\in\mathcal{S}(L)$, its \textit{closure}, denoted by $\overline{S}$, is the smallest closed sublocale containing $S$ and its \textit{interior}, denoted by $\inte(S)$, is the largest open sublocale contained in $S$.
	By a \textit{dense} sublocale we mean a sublocale $S$ of $L$ such that $\overline{S}=L$. In fact, a sublocale is dense if and only if it contains the bottom element of the locale. The sublocale $\mathfrak{B}(L)=\{x\rightarrow 0:x\in L\}$ of a locale $L$ is the \textit{smallest dense sublocale} of $L$ and is referred to as the Booleanization of $L$.
	In classical topology, a subset of a topological space is \textit{nowhere dense} if the interior of its closure is empty. By a \textit{nowhere dense} sublocale we mean a sublocale that misses the smallest dense sublocale and by a \textit{regular-closed} sublocale we refer to a sublocale which is the closure of some open sublocale. We denote by $\Nd(L)$ the join of all nowhere dense sublocales of a locale $L$. We shall use the prefix $S$- for localic properties defined on a sublocale $S$ of $L$. 
	
	By a \textit{nucleus} we mean a function $\nu:L\rightarrow L$ such that (1) $a\leq \nu(a)$,
	(2) $a\leq b$ $\Longrightarrow$ $\nu(a)\leq\nu(b)$,
	(3) $\nu\nu(a)=\nu(a)$, and
	(4) $\nu(a\wedge b)=\nu(a)\wedge\nu(b)$ for every $a,b\in L$. For each sublocale $S\subseteq L$ there is an onto frame homomorphism $\nu_{S}:L\rightarrow S$ defined by $$\nu_{S}(a)={\bigwedge}\{s\in S: a\leq s\}.$$ Open sublocales and closed sublocales of a sublocale $S$ of $L$ are given in terms of nucleus as $$\mathfrak{o}_{S}(\nu_{S}(a))=S\cap \mathfrak{o}(a)\quad\text{and}\quad {\mathfrak{c}_{S}}(\nu_{S}(a))=S\cap \mathfrak{c}(a),$$ respectively, for $a\in L$. For any $S\in\mathcal{S}(L)$ and $x\in L$, $S\subseteq\mathfrak{o}(x)$ if and only if $\nu_{S}(x)=1$.
	
	A localic map $f:L\rightarrow M$ induces the following functions: (1) $f[-]:\mathcal{S}(L)\rightarrow \mathcal{S}(M)$ given by the set-theoretic image of each sublocale of $L$ under $f$, and (2)  $f_{-1}[-]:\mathcal{S}(M)\rightarrow\mathcal{S}(L)$ given by $$f_{-1}[T]={\bigvee}\{A\in \mathcal{S}(L):A\subseteq f^{-1}(T)\}.$$ For a localic map $f:L\rightarrow M$, $x\in M$ and $A\in\mathcal{S}(L)$,  $$f_{-1}[\mathfrak{c}_{M}(x)]=\mathfrak{c}_{L}(f^{*}(x));\quad f_{-1}[\mathfrak{o}_{M}(x)]=\mathfrak{o}_{L}(f^{*}(x))\quad\text{and}\quad f[\overline{A}]\subseteq \overline{f[A]}.$$By an \textit{open} localic map we refer to a localic map $f:L\rightarrow M$ such that $f[A]$ is open for each open $A\in\mathcal{S}(L)$.
	
	We denote by $\mathfrak{O}X$ the locale of open subsets of a topological space $X$, and denote by $\widetilde{A}$ or $S_{A}$ a sublocale of $\mathfrak{O}X$ \textit{induced} by a subset $A$ of a topological space $X$. For a topological space $X$ and $A,B\subseteq X$, we have that $\widetilde{A}\vee\widetilde{B}=\widetilde{A\cup B}$ and $\widetilde{A\cap B}\subseteq\widetilde{A}\cap \widetilde{B}$.
	
	\section{Maximal nowhere dense sublocales}\label{mndsection} 
	
	Veksler \cite{V} says that a closed nowhere dense subset of a Tychonoff space  is \textit{maximal nowhere dense} if it is not nowhere dense in any other closed nowhere dense subset of the space. We broaden our study to arbitrary nowhere dense subsets of any topological space. We give the following definition.
	
	\begin{definition}\label{mnddefspaces}
		A nowhere dense subset $N$ of a topological space $X$ is \textit{maximal nowhere dense} in case there is no nowhere dense subset $K$ of $X$ such that $N$ is nowhere dense in $K$.
	\end{definition}
We aim to introduce maximal nowhere dense sublocales such that a subset of a topological space $X$ is maximal nowhere dense if and only if the sublocale it induces is maximal nowhere dense in $\mathfrak{O}X$. 
	
	We transfer maximal nowhere density to locales by replacing subsets with sublocales from Definition \ref{mnddefspaces}. 
	
	\begin{definition}\label{mnddef}
		Let $L$ be a locale. A nowhere dense sublocale $N$ of $L$ is \textit{maximal nowhere dense (m.n.d)} if there is no nowhere dense sublocale $S$ of $L$ such that $N$ is nowhere dense in $S$.
	\end{definition}
	
	We consider some examples. We remind the reader that a locale $L$ is nowhere dense as a sublocale of itself if and only if $L=\{1\}$. 
	\begin{example}\label{mndexample}
		(1)	In a non-Boolean strongly submaximal locale $L$ (according to \cite{D2}, a locale is \textit{strongly submaximal} if each of its dense sublocales is open), $$\Nd(L)=\bigvee\{N\in\mathcal{S}(L):N\text{ is nowhere dense in }L\}$$ is maximal nowhere dense. To see this, we start by showing that in a strongly submaximal locale $L$, $\Nd(L)$ is nowhere dense. Indeed, observe that in a strongly submaximal locale $L$, the dense sublocale $\mathfrak{B}L$ is open (in particular, complemented), so \begin{align*}
			\mathfrak{B}L\cap \Nd(L)&\quad=\quad\mathfrak{B}L\cap \bigvee\{S: S\ \text{is nowhere dense}\}\\
			&\quad=\quad\bigvee\left\{\mathfrak{B}L\cap S:S\ \text{is nowhere dense}\right\}\\
			&\quad=\quad\bigvee\{\mathsf{O}\}=\mathsf{O}
		\end{align*} 
		making $\Nd(L)$ nowhere dense.
		
		Now, let $A\in\mathcal{S}(L)$ be nowhere dense such that $\Nd(L)$ is nowhere dense in $A$. By the nature of $\Nd(L)$, $A=\Nd(L)$, making $\Nd(L)$ nowhere dense as a sublocale of itself. Hence $\Nd(L)=\mathsf{O}$. This means that $\mathsf{O}$ is the only nowhere dense sublocale of $L$, which contradicts that $L$ is non-Boolean.
		
		An easy computation may show that $\Nd(L)=L\smallsetminus\mathfrak{B}L$ whenever $\Nd(L)$ is nowhere dense in a locale $L$. So, in a non-Boolean strongly submaximal locale $L$, $L\smallsetminus\mathfrak{B}L$ is maximal nowhere dense.

		(2) $\mathsf{O}$ is not maximal nowhere dense. This follows since $\mathsf{O}$ is nowhere dense as a sublocale of itself. As a result, we get that a Boolean locale does not have a maximal nowhere dense sublocale. This also tells us that $N$ in Definition \ref{mnddef} cannot be open. Otherwise, $L\smallsetminus\overline{N}$ is dense (since a sublocale $N$ is nowhere dense if and only if $L\smallsetminus\overline{ N}$ is nowhere dense, see \cite{N}), making $L\smallsetminus N$ dense because $L\smallsetminus\overline{N}\subseteq L\smallsetminus N$. Since $N$ is non-void (since it is maximal nowhere dense, making it different from $\mathsf{O}$) and open, we must have that $(L\smallsetminus N)\cap N\neq\mathsf{O}$ which is not possible because $N$ is complemented.
	\end{example}
	
	In Proposition \ref{mndcmnd} below, we give a characterization of maximal nowhere dense sublocales some part of which will be used in calculations that involve maximal nowhere dense sublocales.
	
	Denote by $x^{*A}$ the pseudocomplement of an $x\in A$, calculated in $A\in \mathcal{S}(L)$. Recall that for a dense $A\in\mathcal{S}(L)$, we have the equalities
	$$x^{*A}=x\rightarrow_{A} 0_{A} =x\rightarrow 0_{L}=x^{*}$$ so that $\mathfrak{B}A=\mathfrak{B}L$.
	
	For use below, we prove the following lemma.
	
	\begin{lemma}\label{mndcmndlemma}
		A sublocale $N$ of a locale $L$ is nowhere dense in $K\in\mathcal{S}(L)$ iff $N$ is nowhere dense in $\overline{K}$.
	\end{lemma}
\begin{proof}
	Recall that $\mathfrak{B}S=\mathfrak{B}L$ for every dense $S\in\mathcal{S}(L)$. Because every sublocale is dense in its closure, we have $N\cap \mathfrak{B}K=N\cap\mathfrak{B}\overline{K}$, which implies $N$ is nowhere dense in $K$ if and only if $N$ is nowhere dense in $\overline{K}$.
\end{proof}
We shall use $\NdS(L)$ to denote the collection of all nowhere dense sublocales of a locale $L$.

Recall from \cite{N} that a sublocale is nowhere dense if and only if its closure is nowhere dense.
	
	\begin{proposition}\label{mndcmnd}
		Let $N$ be a nowhere dense sublocale of a locale $L$. The following statements are equivalent.
		\begin{enumerate}
			\item $N$ is maximal nowhere dense.
			\item There is no closed nowhere dense sublocale of $L$ having $N$ as its nowhere dense sublocale.
			\item $\overline{N}$ is maximal nowhere dense.
			\item $\overline{N}$ is not a nowhere dense sublocale of any closed nowhere dense sublocale of $L$.
			\item There is no dense $y\in L$ such that $y\leq \bigwedge N$ and $\left(\bigwedge N\right)^{*\mathfrak{c}(y)}= y$.
		\end{enumerate}  
	\end{proposition}
	\begin{proof}
		$(1)\Longrightarrow (2)$: Follows since there is no nowhere dense (particularly, closed nowhere dense) sublocale $A$ of $L$ in which $N\in\NdS(A)$.
		
		$(2)\Longrightarrow(3)$: Nowhere density of $\overline{N}$ follows since a sublocale is nowhere dense if and only if its closure is nowhere dense, see \cite{N}. 
		
		Now, suppose that there is a nowhere dense sublocale $K$ such that $\overline{N}$ is nowhere dense in $K$. Since every sublocale of a nowhere dense sublocale is nowhere dense, $N$ is nowhere dense in $K$. By Lemma \ref{mndcmndlemma}, $N$ is nowhere dense in the closed nowhere dense sublocale $\overline{K}$, which contradicts the hypothesis in condition (2). Thus $\overline{N}$ is maximal nowhere dense.

		$(3)\Longrightarrow(4)$: Trivial.

		$(4)\Longrightarrow(5)$: Let $y\in L$ be dense such that $y\leq\bigwedge N$ and $\left(\bigwedge N\right)^{*\mathfrak{c}(y)}= y$. This means that $\bigwedge N$ is $\mathfrak{c}(y)$-dense which implies that $\mathfrak{c}_{\mathfrak{c}(y)}\left(\bigwedge N\right)$ is $\mathfrak{c}(y)$-nowhere dense. Because $\mathfrak{c}_{\mathfrak{c}(y)}\left(\bigwedge N\right)=\overline{N}\cap\mathfrak{c}(y)=\overline{N}$, we have that $\overline{N}$ is nowhere dense in the closed sublocale $\mathfrak{c}(y)$, which contradicts the hypothesis.

		$(5)\Longrightarrow(1)$: Assume that $N$ is nowhere dense in a nowhere dense sublocale $K$ of $L$. By Lemma \ref{mndcmndlemma}, $N$ is nowhere dense in the closed nowhere dense sublocale $\overline{K}$. Therefore $\bigwedge K$ is a dense element of $L$ such that $\bigwedge K\leq \bigwedge N$ and $\left(\bigwedge N\right)^{*\overline{K}}= \bigwedge K$, which is a contradiction.
	\end{proof}
	In terms of closed nowhere dense sublocales, we have the following result which holds since for every $a,b\in L$, \begin{align*}
		a^{*\mathfrak{c}(b)}&\quad=\quad\bigvee_{\mathfrak{c}(b)}\{x\in\mathfrak{c}(b):a\wedge x=0_{\mathfrak{c}(b)}=b\}\\
		&\quad=\quad\nu_{\mathfrak{c}(b)}\left(\bigvee\{x\in\mathfrak{c}(b):a\wedge x=b\}\right)\\
		&\quad=\quad b\vee \left(\bigvee\{x\in L:a\wedge x=b\}\right)\\
		&\quad=\quad b\vee (a\rightarrow b)\\
		&\quad=\quad a\rightarrow b.
	\end{align*}
	\begin{corollary}
		Let $L$ be a locale and $\mathfrak{c}(x)\in\mathcal{S}(L)$. Then $\mathfrak{c}(x)$ is maximal nowhere dense if and only if $x$ is dense and there is no dense $y\in L$ such that $y\leq x$ and $x\rightarrow y=y$.
	\end{corollary}
	
	Proposition \ref{mndcmnd} suggests that, when doing calculations about maximal nowhere dense sublocales, there is no loss of generality with restricting to closed nowhere dense sublocales.
	
	We give the following result regarding binary intersections of induced sublocales which will be used below. Recall that for any open subset $U$ of a topological space $X$, $\mathfrak{o}(U)=\widetilde{U}$ and $\mathfrak{O}X\smallsetminus\widetilde{U}=\widetilde{X\smallsetminus U}$.
	
	\begin{lemma}\label{binaryintersection}
		Let $X$ be a topological space. For any $A, B\subseteq X$ with $\widetilde{B}$ complemented in $\mathfrak{O}X$, $\widetilde{A\cap B}=\widetilde{A}\cap \widetilde{B}$.
	\end{lemma} 
	\begin{proof} It is clear that $\widetilde{A\cap B}\subseteq\widetilde{A}\cap \widetilde{B}$.
		
		On the other hand, $A\cap B\subseteq A\cap B$ implies that $A\subseteq (A\cap B)\cup (X\smallsetminus B)$. Because $\widetilde{S\cup T}=\widetilde{S}\vee \widetilde{T}$ for all $S,T\subseteq X$, we have that $\widetilde{A}\subseteq\widetilde{A\cap B}\vee \widetilde{X\smallsetminus B}=\widetilde{A\cap B}\vee (\mathfrak{O}X\smallsetminus\widetilde{B})$. Therefore $$\widetilde{A}\cap\widetilde{B}=\widetilde{A}\cap (\mathfrak{O}X\smallsetminus(\mathfrak{O}X\smallsetminus\widetilde{B}))\subseteq \widetilde{A\cap B}$$ since $\widetilde{B}$ is complemented.
	\end{proof}

	We show below that a subset is nowhere dense in a subspace of a T$_{D}$-space precisely when the sublocale it induces is nowhere dense in the sublocale induced by the subspace. We shall make use of \cite[Proposition 4.1.]{BP1} which states that a sublocale is nowhere dense if and only if its closure has a void interior. We remind the reader that open sublocales of a sublocale $S$ of a locale $L$ are the $\mathfrak{o}_{S}(a)=S\cap \mathfrak{o}(a)$ for $a\in S$. So, in $\mathfrak{O}X$, open sublocales of $S\in\mathcal{S}(\mathfrak{O}X)$ are the $\mathfrak{o}_{S}(U)=S\cap\mathfrak{o}(U)=S\cap\widetilde{U}$ for $U\in S$.
	
	We give the following lemma which we shall use below.
	
	\begin{lemma}\label{N1}
		Let $X$ be a T$_{D}$-space. For each subset $N$ such that $\widetilde{N}$ is complemented in $\mathfrak{O}X$,  $N\cap A=\emptyset$ if and only if $\widetilde{N}\cap \widetilde{A}=\mathsf{O}$. 
	\end{lemma}
\begin{proof}
	 $(\Longrightarrow)$: Follows from Lemma \ref{binaryintersection}.
	 
	$(\Longleftarrow)$: Follows from the fact that $B\subseteq C$ if and only if $\widetilde{B}\subseteq\widetilde{C}$, for all $B,C\subseteq X$ where $X$ is a T$_{D}$-space. See \cite{PP1}.
\end{proof}
	
\begin{lemma}\label{lembinaryintersection}
	Let $X$ be a T$_{D}$-space and $F\subseteq X$. Then $A\subseteq F$ is $F$-nowhere dense iff $\widetilde{A}$ is $\widetilde{F}$-nowhere dense.
\end{lemma}
\begin{proof}
	$(\Longrightarrow):$ Let $U\in \widetilde{F}$ be such that $\mathfrak{o}_{\widetilde{F}}(U)\subseteq \overline{\widetilde{A}}^{\widetilde{F}}$. Then $U\in\mathfrak{O}X$ and $\widetilde{U}\cap \widetilde{F}\subseteq\overline{\widetilde{A}}^{\widetilde{F}}=\overline{\widetilde{A}}\cap \widetilde{F}$ so that $\widetilde{U\cap F}\subseteq\widetilde{U}\cap\widetilde{F}\subseteq\widetilde{\overline{A}}$. Therefore $U\cap F\subseteq \overline{A}$ making $U\cap F\subseteq \overline{A}^{F}$. Since $A$ is $F$-nowhere dense, $U\cap F=\emptyset$. By Lemma \ref{N1}, $\widetilde{U}\cap\widetilde{F}=\mathsf{O}$ making $\Int_{\widetilde{F}}\left(\overline{\widetilde{A}}^{\widetilde{F}}\right)=\mathsf{O}$. Thus $\widetilde{A}$ is $\widetilde{F}$-nowhere dense.

	$(\Longleftarrow):$ Let $U\in\mathfrak{O}X$ be such that $U\cap F\subseteq \overline{A}^{F}=\overline{A}\cap F$. Then $U\cap F\subseteq \overline{A}$. Since $U$ is open, it follows from Lemma \ref{binaryintersection} that $\widetilde{U}\cap\widetilde{F}\subseteq\widetilde{\overline{A}}$ which gives $\widetilde{U}\cap\widetilde{F}\subseteq\widetilde{\overline{A}}\cap\widetilde{F}$. Because $\widetilde{A}$ is $\widetilde{F}$-nowhere dense, $\widetilde{U}\cap\widetilde{F}=\mathsf{O}$ implying that $U\cap F=\emptyset$. Therefore $A$ is $F$-nowhere dense.
\end{proof}We are now in a position to show that the notion of maximal nowhere density introduced in Definition \ref{mnddef} is conversative in locales. Recall that for a subset $N$ of a topological space $X$:
\begin{enumerate}
	\item $N$ is nowhere dense in $X$ if and only if $\widetilde{N}$ is nowhere dense in $\mathfrak{O}X$, and
	\item If $X$ is a T$_{D}$-space, then $\overline{\widetilde{N}}=\widetilde{\overline{N}}$.
\end{enumerate}

\begin{proposition}\label{mnd}
	Let $X$ be a T$ _{D} $-space. A subset $F$ of $X$ is maximal nowhere dense in $X$ iff $\widetilde{F}$ is maximal nowhere dense in $\mathfrak OX$.
\end{proposition}
\begin{proof}
	$(\Longrightarrow)$: Let $F\subseteq X$ be such that $\widetilde{F}$ is not maximal nowhere dense in $\mathfrak OX$. It follows from Proposition \ref{mndcmnd} that the sublocale $\overline{\widetilde{F}}=\widetilde{\overline{F}}$ is not maximal nowhere dense in $\mathfrak{O}X$. Therefore, there is a closed sublocale $\mathcal{K}$ of $\mathfrak{O}X$ such that $\widetilde{\overline{F}}$ in nowhere dense in $\mathcal{K}$.  Because $\mathcal{S}$ is a closed sublocale of $\mathfrak OX$, choose a closed set $K\subseteq X$ such that $\mathcal{K}=\mathfrak{c}(X\smallsetminus K)=\widetilde{K}$. It follows from Lemma \ref{lembinaryintersection} that $\overline{F}$ is $K$-nowhere dense, making $F$ not maximal nowhere dense in $X$.

	$ (\Longleftarrow) $: If $F$ is not maximal nowhere dense in $X$, then $F$ is nowhere dense in some nowhere dense subset $K$ of $X$. It follows from Lemma \ref{lembinaryintersection} that $\widetilde{F}$ is nowhere dense in the nowhere dense sublocale $\widetilde{K}$ of $\mathfrak{O}X$. Thus $\widetilde{F}$ is not maximal nowhere dense in $\mathfrak{O}X$.
\end{proof}
	
	We close this section by considering some results about maximal nowhere dense sublocales.
	
	Recall from \cite{FPP} that for any sublocale $S$ of $L$, $$\Int(S)=\mathfrak{o}\left(\bigwedge (L\smallsetminus S)\right)=L\smallsetminus \mathfrak{c}\left(\bigwedge (L\smallsetminus S)\right)=L\smallsetminus \overline{L\smallsetminus S}.$$ This can be applied to any complemented sublocale $A$ of a locale $L$ and arbitrary $F\in\mathcal{S}(L)$ as follows:
	\begin{align*}
		\Int_{A}\left(\overline{F}\cap A\right) 
		&\quad =\quad
		A\smallsetminus \overline{(A\smallsetminus \left(\overline{F}\cap A\right))}^{A}   \\
		&\quad =\quad
		A\smallsetminus \left(\overline{A\smallsetminus \left(\overline{F}\cap A\right)}\cap A\right)  \\
		&\quad =\quad
		A\cap \left(L\smallsetminus \left(\overline{A\smallsetminus \left(\overline{F}\cap A\right)}\cap A\right) \right)  \\
		&\quad =\quad
		A\cap\left( \left(L\smallsetminus \left(\overline{A\smallsetminus \left(\overline{F}\cap A\right)}\right)\right)\vee (L\smallsetminus A) \right)  \\
		&\quad =\quad
		\left(A\cap\left(L\smallsetminus \left(\overline{A\smallsetminus \left(\overline{F}\cap A\right)}\right)\right)\right)\vee (A\cap (L\smallsetminus A))   \\
		&\quad =\quad
		A\cap\left(L\smallsetminus \left(\overline{A\smallsetminus \left(\overline{F}\cap A\right)}\right)\right) \\
		&\quad =\quad
		A\cap\left(L\smallsetminus \overline{\left(A\cap \left(L\smallsetminus \left(\overline{F}\cap A\right)\right)\right)}\right)  \\
		&\quad =\quad
		A\cap \left(L\smallsetminus \overline{\left(A\cap \left(L\smallsetminus \overline{F}\right)\right)\vee \left(A\cap\left(L\smallsetminus A\right)\right)}\right)  \\
		&\quad =\quad
		A\cap \left(L\smallsetminus \overline{\left(A\cap \left(L\smallsetminus \overline{F}\right)\right)}\right).
	\end{align*}
	
	\begin{obs}\label{obsndsubspace}
		From the preceding paragraph we get that a sublocale $F$ of a complemented sublocale $A$ of $L$ is $A$-nowhere dense if and only if $A\subseteq \overline{A\cap (L\smallsetminus \overline{F})}$. This follows since \begin{align*}
			F\text{ is nowhere dense in } A&\quad\Longleftrightarrow\quad\Int_{A}\left(\overline{F}^{A}\right)=\mathsf{O}\\
			&\quad\Longleftrightarrow\quad A\cap \left(L\smallsetminus \overline{\left(A\cap \left(L\smallsetminus \overline{F}\right)\right)}\right)=\mathsf{O}\\
			&\quad\Longleftrightarrow\quad  A\subseteq\overline{A\cap \left(L\smallsetminus \overline{F}\right)}.
		\end{align*}
	\end{obs}
	\begin{proposition}\label{mndprop}
		Let $L$ be a locale and $F$ a non-void nowhere dense sublocale of $L$. Then
		\begin{enumerate}
			\item If $A\in\NdS(L)$, $F$ is maximal nowhere dense in $L$ and $F\subseteq A$, then $A$ is maximal nowhere dense in $L$.
			\item If $F\cap (L\smallsetminus\overline{N})\neq\mathsf{O}$ for all  $N\in\NdS(L\smallsetminus F)$, then $F$ is maximal nowhere dense in $L$.
		\end{enumerate}
	\end{proposition}
	\begin{proof}
		(1) If $A\in\NdS(L)$, $F$ is maximal nowhere dense in $L$, $F\subseteq A$ and $N\in\NdS(L)$ such that $A\in\NdS(N)$, then $F\in\NdS(N)$, which is a contradiction. Thus $A$ is maximal nowhere dense in $L$.

		(2) We prove this statement by contradiction: Suppose that $F$ is not maximal nowhere dense, i.e., there is a nowhere dense sublocale $\mathfrak{c}(x)$ such that $F\in\NdS(\mathfrak{c}(x))$. We get that $\mathfrak{c}(x)\cap \left(L\smallsetminus\overline{\mathfrak{c}(x)\cap(L\smallsetminus F)}\right)= \mathsf{O}$. 
		
		\underline{Claim: $\mathfrak{c}(x)\cap(L\smallsetminus F)\in\NdS(L\smallsetminus F)$}. To verify this, assume that $\mathfrak{c}(x)\cap(L\smallsetminus F)$ is not $(L\smallsetminus F)$-nowhere dense. Then there is a non-void $(L\smallsetminus F)$-open sublocale $S$ such that $S\subseteq \overline{(L\smallsetminus F)\cap\mathfrak{c}(x)}$. Such $S$ is of the form $S=\mathfrak{o}(a)\cap(L\smallsetminus F)$ for some $a\in L$. Therefore  $$\mathfrak{o}(a)\cap (L\smallsetminus\overline{F})\subseteq\mathfrak{o}(a)\cap (L\smallsetminus F)\subseteq \overline{(L\smallsetminus F)\cap\mathfrak{c}(x)} \subseteq \overline{(L\smallsetminus F)}\cap\overline{\mathfrak{c}(x)}\subseteq\mathfrak{c}(x).$$ But $\mathfrak{c}(x)\in\NdS(L)$ and $\mathfrak{o}(a)\cap(L\smallsetminus\overline{F})$ is open in $L$, so $\mathfrak{c}(a)\cap(L\smallsetminus\overline{F})=\mathsf{O}$. Therefore $\mathfrak{o}(a)\subseteq \overline{F}$. Since $\overline{F}$ is nowhere dense in $L$, $\mathfrak{o}(a)=\mathsf{O}$ making $\mathsf{O}=S=\mathfrak{o}(a)\cap (L\smallsetminus F)$, which is impossible.  Thus $\mathfrak{c}(x)\cap(L\smallsetminus F)$ is $(L\smallsetminus F)$-nowhere dense.
		
		Now, by hypothesis, $F\cap (L\smallsetminus\overline{\mathfrak{c}(x)\cap (L\smallsetminus F)})\neq \mathsf{O}$. Since $F\subseteq\mathfrak{c}(x)$, $\mathfrak{c}(x)\cap (L\smallsetminus\overline{(L\smallsetminus F)\cap\mathfrak{c}(x)})\neq \mathsf{O}$ which contradicts that $\mathfrak{c}(x)\cap \left(L\smallsetminus\overline{\mathfrak{c}(x)\cap(L\smallsetminus F)}\right)= \mathsf{O}$. Thus $F$ is maximal nowhere dense in $L$.
	\end{proof}
	\begin{obs}
		Using the fact that non-empty finite joins of nowhere dense sublocales are nowhere dense, Proposition \ref{mndprop}(1) tells us that any finite join of maximal nowhere dense sublocales is maximal nowhere dense.
	\end{obs}

	As far as maximality of objects with some property P is concerned, one expects it to refer to an object $A$ with property P in which no other object with property P, other than $A$ itself, contains $A$. So, we spare some time to show that a nowhere dense sublocale with this maximality exists precisely when a locale has a largest nowhere dense sublocale. 
	
	Let us call a nowhere dense sublocale $N$ of a locale $L$ \textit{strongly maximal nowhere dense} if, for any nowhere dense sublocale $A$, $N\subseteq A$ implies $A=N$.

	\begin{proposition}
		Let $L$ be a locale. The following statements are equivalent.
		\begin{enumerate}
			\item $L$ has a strongly maximal nowhere dense sublocale.
			\item $\Nd(L)$ is nowhere dense.
		\end{enumerate}  	
			 If $L$ is non-Boolean, this is further equivalent to:
			\begin{enumerate}
				\item[(3)] $\Nd(L)$ is maximal nowhere dense.	
			\end{enumerate} 
	\end{proposition}
	\begin{proof}
		$(1)\Longrightarrow(2)$: Let $A\in\mathcal{S}(L)$ be strongly maximal nowhere dense. We show that $\Nd(L)\subseteq A$ which will make $\Nd(L)$ nowhere dense. Choose a nowhere dense $N\in\mathcal{S}(L)$. Then $N\vee A$ is nowhere dense in $L$. But $A\subseteq N\vee A$ and $A$ is strongly maximal nowhere dense, so $A=N\vee A$. Therefore $N\subseteq N\vee A=A$. Since $N$ was arbitrary, $\Nd(L)=\bigvee\{S\in\mathcal{S}(L):S\in\NdS(L)\}\subseteq A$. Thus $\Nd(L)\subseteq A$ implying that $\Nd(L)$ is nowhere dense.

$(2)\Longrightarrow(1)$: If $\Nd(L)$ is nowhere dense, then there is no other nowhere dense sublocale containing $\Nd(L)$ other than itself. Thus $\Nd(L)$ is a strongly maximal nowhere dense sublocale of $L$. 

Assume that $L$ is non-Boolean. The equivalence $(2)\Longleftrightarrow(3)$ follows since $\Nd(L)$ contains every nowhere dense sublocale of $L$ and $\Nd(L)\neq\mathsf{O}$.
	\end{proof}

	
	\section{Homogeneous maximal nowhere dense sublocales}
	
	Related to maximal nowhere dense subsets are homogeneous maximal nowhere dense subsets which were defined for spaces by Veksler in \cite{V} as closed nowhere dense subsets $F$ of a Tychonoff space $X$ in which each non-empty $F$-regular-closed subset is maximal nowhere dense in $X$. In this paper, we do not only focus on Tychonoff spaces, but all topological spaces.

	We extend Veksler's definition of a homogeneous maximal nowhere dense subset of any topological space to locales as follows.
	
	\begin{definition}\label{hmnd}
		A closed nowhere dense sublocale $N$ of a locale $L$ is \textit{homogeneous maximal nowhere dense (h.m.n.d)} if each non-void $N$-regular-closed sublocale is maximal nowhere dense in $L$.
	\end{definition}
	Without the closedness requirement in Definition \ref{hmnd}, we shall call $F$ \textit{weakly homogeneous maximal nowhere dense}.
	
	
\begin{obs}\label{regularclosed}
	We note that regular-closed sublocales of a closed nowhere dense sublocale $\mathfrak{c}(x)$ of $L$ are of the form $\mathfrak{c}(a\rightarrow x)$ for some $a\in L$. Indeed, $A$ is $\mathfrak{c}(x)$-regular-closed if and only if $A=\overline{\mathfrak{o}(a)\cap \mathfrak{c}(x)}\cap \mathfrak{c}(x)$ for some $a\in L$. We have \begin{align*}
		\overline{\mathfrak{o}(a)\cap \mathfrak{c}(x)}\cap\mathfrak{c}(x)&\quad=\quad\mathfrak{c}\left(\bigwedge\left(\mathfrak{o}(a)\cap\mathfrak{c}(x)\right)\right)\cap\mathfrak{c}(x)\\
		&\quad=\quad\mathfrak{c}\left(\bigwedge\left(\mathfrak{c}_{\mathfrak{o}(a)}(\nu_{\mathfrak{o}(a)}(x))\right)\right)\cap\mathfrak{c}(x)\\
		&\quad=\quad\mathfrak{c}\left(\nu_{\mathfrak{o}(a)}(x)\right)\cap\mathfrak{c}(x)\\
		&\quad=\quad\mathfrak{c}\left(a\rightarrow x\right)\cap\mathfrak{c}(x)\\
		&\quad=\quad\mathfrak{c}\left(a\rightarrow x\right)\quad\text{since }x\leq a\rightarrow x.
	\end{align*}
\end{obs}In light of Observation \ref{regularclosed} and the characterizations of maximal nowhere nowhere dense sublocales given in Proposition \ref{mndcmnd}, we get the following characterizations of homogeneous maximal nowhere dense sublocales. The proof is straight forward and shall be omitted.
	\begin{proposition}\label{hmndcharac}
		Let $L$ be a locale and $x\in L$. The following statements are equivalent.
		\begin{enumerate}
			\item $\mathfrak{c}(x)$ is homogeneous maximal nowhere dense.
			\item For each $a\in L$, $\mathfrak{c}(a\rightarrow x)$ is maximal nowhere dense in $L$.
			\item For each $a\in L$, there is no dense $y\in L$ such that $y\leq a\rightarrow x$ and $(a\rightarrow x)^{*\mathfrak{c}(y)}=y$.
		\end{enumerate}
	\end{proposition}
	
	In the next result, we show that the definition of a homogeneous maximal nowhere dense sublocale given in Definition \ref{hmnd} is conservative in locales. Prior to that, we give the following lemmas where the proof for Lemma \ref{Knd}(1) is a straightforward application of Lemma \ref{binaryintersection} and Lemma \ref{lembinaryintersection}, and shall be omitted. We mentioned in the Preliminaries section that we shall sometimes use the notation $S_{A}$ instead of $\widetilde{A}$ for the sublocale of $\mathfrak{O}X$ induced by the subset $A$ of a topological space $X$.
	
\begin{lemma}\label{Knd}
	Let $X$ be a T$_{D}$-space and $A,K,F, U\subseteq X$ be such that $\widetilde{F}$ is complemented in $\mathfrak{O}X$. Then:
	\begin{enumerate}
		\item If $U\cap F\subseteq K$,  then $U\cap F$ is $K$-nowhere dense iff $\widetilde{U}\cap\widetilde{F}$ is $\widetilde{K}$-nowhere dense.
		\item $A=\overline{F\cap U}\cap F$ if and only if $\widetilde{A}=\overline{\widetilde{F}\cap\widetilde{U}}\cap\widetilde{F}$.
	\end{enumerate} 
\end{lemma}
\begin{proof}
	(2) We have that \begin{align*}
		A=\overline{F\cap U}\cap F &\quad\Longleftrightarrow \quad \widetilde{A}=S_{\overline{F\cap U}\cap F}\\
		&\quad\Longleftrightarrow \quad \widetilde{A}=\widetilde{\overline{F\cap U}}\cap\widetilde{F}\text{ since }\widetilde{F} \text{ is complemented in }\mathfrak{O}X\\
		&\quad\Longleftrightarrow \quad \widetilde{A}=\overline{\widetilde{F\cap U}}\cap\widetilde{F}\\
		&\quad\Longleftrightarrow \quad \widetilde{A}=\overline{\widetilde{F}\cap\widetilde{U}}\cap\widetilde{F}\text{ since }\widetilde{F} \text{ is complemented in }\mathfrak{O}X
	\end{align*}which proves the result.
\end{proof}
	Using Lemma \ref{binaryintersection}(2), we get the following result, whose proof shall be omitted.
 	\begin{corollary}\label{regularclosedF}
		Let $X$ be a T$_{D}$-space and $F\subseteq X$. Then a subset $A\subseteq F$ is $F$-regular-closed iff $\widetilde{A}$ is $\widetilde{F}$-regular-closed.
	\end{corollary}
\begin{proposition}
	Let $X$ be T$ _{D} $-space. A closed set $F\subseteq X$ is homogeneous maximal nowhere dense iff $\widetilde{F}\subseteq\mathfrak OX$ is homogeneous maximal nowhere dense.
\end{proposition}

\begin{proof}
	Follows from Proposition \ref{mnd}, Lemma \ref{Knd}(1) and Corollary \ref{regularclosedF}.
\end{proof}
	
	The following result tells us that homogeneous maximal nowhere density is regular-closed hereditary.
	\begin{proposition}
		Let $L$ be a locale and $F$ be a closed nowhere dense sublocale of $L$. If $F$ is homogeneous maximal nowhere dense and $A$ is a non-void $F$-regular-closed sublocale, then $A$ is homogeneous maximal nowhere dense.
	\end{proposition}
	\begin{proof}
		Let $N$ be a non-void regular-closed sublocale of $A$ and suppose that there is $B\in\NdS(L)$ such that $N\in\NdS(B)$. Because $A$ is $F$-regular-closed and $N$ is $A$-regular-closed, $A=\overline{\mathfrak{o}(x)\cap F}\cap F$ and $N=\overline{\mathfrak{o}(y)\cap A}\cap A$ for some $x,y\in L$.  Since both $F$ and $A$ are closed, $A=\overline{\mathfrak{o}(x)\cap F}$ and $N=\overline{\mathfrak{o}(y)\cap A}$ so that $N=\overline{\mathfrak{o}(y)\cap \overline{\mathfrak{o}(x)\cap F}}$. Therefore $$\overline{\mathfrak{o}(y)\cap\mathfrak{o}(x)\cap F}=\overline{\mathfrak{o}(y\wedge x)\cap F}\subseteq N.$$ The sublocale $\overline{\mathfrak{o}(y)\cap\mathfrak{o}(x)\cap F}\neq\mathsf{O}$, otherwise $\mathfrak{o}(y)\cap\mathfrak{o}(x)\cap F=\mathsf{O}$ making $\mathfrak{o}(y)\cap\overline{\mathfrak{o}(x)\cap F}=\mathsf{O}$ which is not possible. Since $N$ is nowhere dense in $B$, we get that $\overline{\mathfrak{o}(y\wedge x)\cap F}$ is nowhere dense in $B$ making $\overline{\mathfrak{o}(y\wedge x)\cap F}\cap F\in\NdS(B)$. This is not possible because $\overline{\mathfrak{o}(y\wedge x)\cap F}\cap F$ is non-void and regular-closed in $F$ which must be maximal nowhere dense in $L$.  Thus $A$ is homogeneous maximal nowhere dense.
	\end{proof}
	
	We close this section by considering a relationship between maximal nowhere dense sublocales and (strongly) homogeneous maximal nowhere dense sublocales.
	
	\begin{proposition}\label{smndandmnd}Every (strongly) homogeneous maximal nowhere dense sublocale is maximal nowhere dense.
	\end{proposition}
	\begin{proof}Follows since every locale is a regular-closed sublocale of itself.
	\end{proof}

	\section{Maximal nowhere density and Inaccessibility}\label{inaccessandmnd}
		The aim of this section is to explore a relationship between the variations of nowhere dense sublocales discussed in the previous two sections (which are maximal nowhere dense sublocales and homogeneous maximal nowhere dense sublocales) and sublocales called inaccessible sublocales and almost inaccessible sublocales.
	
	En route to introducing inaccessible and almost inaccessible sublocales, we consider Veksler's notions of an inaccessible point and an almost inaccessible point of a Tychonoff space which he introduced in \cite{V}. He calls a point $x\in E\subseteq X$, where $X$ is a Tychonoff space, $E$-inaccessible (resp. almost $E$-inaccessible) if $x\notin \overline{N}$ (resp. $x\notin\Int_{E}(\overline{N}\cap E)$) for all $(X\smallsetminus E)$-closed nowhere dense $N$. We shall transfer these notions to locales.

	Our journey to introducing localic notions of inaccessible and almost inaccessible points will start with inaccessible points and end with almost inaccessible points.
	
	In a Tychonoff space $X$, we have that $x\notin \overline{N}$ if and only if $\overline{\{x\}}\cap \overline{N}=\emptyset$ if and only if $(X\smallsetminus\overline{\{x\}})\cup (X\smallsetminus\overline{N})=X$ for every $x\in X$, $N\subseteq X$.  So the definition of an $E$-inaccessible point $x\in E\subseteq X$ is equivalent to:
	
	\begin{quote}
		\emph{$(X\smallsetminus\overline{\{x\}})\cup (X\smallsetminus\overline{N})=X$ for all $(X\smallsetminus E)$-closed nowhere dense $N$.} 
	\end{quote}
	Recall that $X\smallsetminus\overline{A}=0_{\widetilde{A}}$ for any subset $A$ of a space $X$. This and the preceding paragraph motivate the following localic definition of an inaccessible point.
	\begin{definition}\label{inaccessible}
		A point $p$ of a sublocale $S$ of a completely regular locale $L$ is \textit{$S$-inaccessible} if for each $(L\smallsetminus S)$-closed nowhere dense sublocale $N$, $0_{N}\vee p=1$, where the join is calculated in $L$.
	\end{definition}

	Recall from \cite[Lemma 3.2.1]{BPPP} that for a subset $A$ of a T$_{D}$-space $X$, $$\bigvee\{\{X\smallsetminus\overline{\{x\}},X\}:x\notin A\}=\widetilde{X\smallsetminus A}$$ is the supplement of $\widetilde{A}$, i.e., $\widetilde{X\smallsetminus A}=\widetilde{X}\smallsetminus\widetilde{A}$. 
	
	According to \cite{DM}, a regular locale is \textit{T$ _{1} $} in the sense that every point is a maximal element. Hence a point $p$ of a regular locale $L$ has a property that $a\vee p=1$ if and only if $a\nleq p$ for every $a\in L$.  
	
	In what follows, we show that a point $x$ of a Tychonoff space $X$ is inaccessible if and only if $\widetilde{x}$ is inaccessible. We shall need the following lemma. We shall make use of the fact that $\widetilde{A}=\{\inte((X\smallsetminus A)\cup G):G\in \mathfrak{O}X\}$ for every subset $A$ of a topological space $X$.
	
	\begin{lemma}\label{lemmainacc}
		Let $X$ be a topological space, $F$ a closed subset of $X$ and $A\subseteq X$. Then $\bigwedge(\widetilde{F}\cap\widetilde{A})=\inte((X\smallsetminus A)\cup (X\smallsetminus F))$
	\end{lemma}\begin{proof}
		We have that $X\smallsetminus F\subseteq (X\smallsetminus A)\cup (X\smallsetminus F)$, making $$X\smallsetminus F=\inte(X\smallsetminus F)\subseteq \inte((X\smallsetminus A)\cup (X\smallsetminus F))$$ so that $\inte((X\smallsetminus A)\cup(X\smallsetminus F))\in\widetilde{F}$. Also, since $X\smallsetminus F$ is open, $\inte((X\smallsetminus A)\cup(X\smallsetminus F))\in \widetilde{A}.$ Therefore $\inte((X\smallsetminus A)\cup(X\smallsetminus F))\in \widetilde{F}\cap \widetilde{A}$, making  $\bigwedge(\widetilde{F}\cap\widetilde{A})\leq\inte((X\smallsetminus A)\cup (X\smallsetminus F))$. 
		
		On the other hand, let $V\in \widetilde{F}\cap\widetilde{A}$. Then $X\smallsetminus F\subseteq V$ and $V=\inte((X\smallsetminus A)\cup G)$ for some $G\in\mathfrak{O}X$. We get that \begin{align*}
			\inte((X\smallsetminus A)\cup (X\smallsetminus F))&\quad\subseteq\quad \inte((X\smallsetminus A)\cup V)\\
			&\quad=\quad \inte\left((X\smallsetminus A)\cup \inte((X\smallsetminus A)\cup G)\right)\\
			&\quad\subseteq\quad \inte((X\smallsetminus A)\cup G)\\
			&\quad=\quad V.
		\end{align*}Therefore $\inte((X\smallsetminus A)\cup (X\smallsetminus F))\leq \bigwedge(\widetilde{F}\cap\widetilde{A})$ and hence  $\bigwedge(\widetilde{F}\cap\widetilde{A})=\inte((X\smallsetminus A)\cup (X\smallsetminus F))$.
	\end{proof}

	\begin{proposition}\label{vnd}
		Let $X$ be a Tychonoff space and $x\in E\subseteq X$. Then $x$ is $E$-inaccessible iff $\widetilde{x}$ is $\widetilde{E}$-inaccessible. 
	\end{proposition}
	\begin{proof}
		$(\Longrightarrow):$ Let $K$ be an $(\widetilde{X}\smallsetminus\widetilde{E})$-closed nowhere dense sublocale. Choose a closed subset $F$ of $X$ such that $K=\widetilde{F}\cap (\widetilde{X}\smallsetminus\widetilde{E})$. Since  $\widetilde{X}\smallsetminus\widetilde{E}=\widetilde{X\smallsetminus E}$, $K=\widetilde{F}\cap\widetilde{X\smallsetminus E}\supseteq S_{\left(F\cap (X\smallsetminus E)\right)}$. It follows from Lemma \ref{lembinaryintersection} that $F\cap (X\smallsetminus E)$ is an $(X\smallsetminus E)$-closed nowhere dense subset. Since $x$ is $E$-inaccessible, $$\left(X\smallsetminus\overline{\{x\}}\right)\cup \left(X\smallsetminus\overline{F\cap (X\smallsetminus E)}\right)=X.$$ Because $\widetilde{x}=X\smallsetminus\overline{\{x\}}$ is a point and every completely regular locale is regular and hence T$_{1}$, it follows that $X\smallsetminus\overline{F\cap (X\smallsetminus E)}\nleq \widetilde{x}$.  Since $X\smallsetminus\overline{F\cap (X\smallsetminus E)}=\inte(E\cup (X\smallsetminus F))$ and $\inte(E\cup (X\smallsetminus F))=\bigwedge(\widetilde{F}\cap\widetilde{X\smallsetminus E})$ by Lemma \ref{lemmainacc}, $\bigwedge(\widetilde{F}\cap\widetilde{X\smallsetminus E})\nleq \widetilde{x}$ so that $$\bigwedge(\widetilde{F}\cap\widetilde{X\smallsetminus E})\vee \widetilde{x}=0_{K}\vee \widetilde{x}=1_{\mathfrak{O}X}$$ because $\mathfrak{O}X$ is T$_{1}$. Thus $\widetilde{x}$ is a $\widetilde{E}$-inaccessible point.
		
		$(\Longleftarrow):$ Let $C$ be an $(X\smallsetminus E)$-closed nowhere dense subset. Set $C=F\cap (X\smallsetminus E)$ for some closed $F\subseteq X$. If follows from Proposition \ref{Knd} that $\widetilde{F}\cap\widetilde{X\smallsetminus E}$ is $(\widetilde{X\smallsetminus E})$-closed nowhere dense. But $\widetilde{X}\smallsetminus\widetilde{E}=\widetilde{X\smallsetminus E}$, so the $(\widetilde{X}\smallsetminus\widetilde{E})$-closed sublocale $\widetilde{F}\cap (\widetilde{X}\smallsetminus\widetilde{E})$ is $(\widetilde{X}\smallsetminus\widetilde{E})$-nowhere dense. By hypothesis, $\widetilde{x}\vee 0_{\widetilde{F}\cap (\widetilde{X}\smallsetminus\widetilde{E})}=1_{\mathfrak OX}$. But $$0_{\widetilde{F}\cap (\widetilde{X}\smallsetminus\widetilde{E})}=\bigwedge\left(\widetilde{F}\cap (\widetilde{X}\smallsetminus\widetilde{E})\right)=\inte(E\cup (X\smallsetminus F))=X\smallsetminus\overline{F\cap (X\smallsetminus E)}$$where the second equality follows from Lemma \ref{lemmainacc}. Therefore $\left(X\smallsetminus\overline{\{x\}}\right)\cup \left(X\smallsetminus\overline{F\cap (X\smallsetminus E)}\right)=X$ which implies that $$\emptyset=\overline{\{x\}}\cap \overline{F\cap (X\smallsetminus E)}=\{x\}\cap \overline{F\cap (X\smallsetminus E)}.$$ Thus $x\notin \overline{F\cap (X\smallsetminus E)}=\overline{C}$, making $x$ $E$-inaccessible.
	\end{proof}

	To transfer almost inaccessibility to locales, we recall that for $x\in E\subseteq X$,
	\begin{align*}
		x\notin \Int_{E}(\overline{N}\cap E) 
		& \quad\Longleftrightarrow \quad
		x\notin E\cap \left(X\smallsetminus\overline{E\cap (X\smallsetminus\overline{N})}\right)\\
		& \quad\Longleftrightarrow \quad
		x\in \overline{E\cap (X\smallsetminus\overline{N})}.
	\end{align*} The above equivalence motivates the following localic definition of an almost inaccessible point. 
	
	\begin{definition}
		A point $p$ of a sublocale $S$ of a completely regular locale $L$ is \textit{almost $S$- inaccessible} if for each $(L\smallsetminus S)$-closed nowhere dense sublocale $N$, $p\in \cl_{S}\left(S\cap\left(L\smallsetminus\overline{ N}\right)\right)$.
	\end{definition}
	In what follows we prove that a point $x$ of $X$ is almost inaccessible precisely when $\widetilde{x}$ is almost inaccessible.
	
	\begin{proposition}
		Let $X$ be a Tychonoff space. A point $x$ of a subset $E$ of $X$ is almost $E$-inaccessible iff $\widetilde{x}$ is almost $\widetilde{E}$-inaccessible.
	\end{proposition}
	\begin{proof}
		$(\Longrightarrow):$ It is clear that $\widetilde{x}$ is a point belonging to the closed sublocale $\widetilde{\overline{E}}=\overline{\widetilde{E}}$. Choose an $(\widetilde{X}\smallsetminus\widetilde{E})$-closed nowhere dense $K$. Then $K=\widetilde{F}\cap (\widetilde{X}\smallsetminus\widetilde{E})$ for some closed $F\subseteq X$. Since $\widetilde{X}\smallsetminus\widetilde{E}=\widetilde{X\smallsetminus E}$, $K=\widetilde{F}\cap \widetilde{X\smallsetminus E}$. It follows from Proposition \ref{Knd} that $F\cap (X\smallsetminus E)$ is $(X\smallsetminus E)$-closed nowhere dense. Therefore $x\notin\inte_{E}(\overline{F\cap (X\smallsetminus E)}\cap E)$ which means $x\in\overline{E\cap (X\smallsetminus \overline{F\cap (X\smallsetminus E)})}$. We show that $\widetilde{x}\in \overline{\widetilde{E}\cap(\widetilde{X}\smallsetminus\overline{K})}$. Let $U\in\mathfrak OX$ be such that $\widetilde{E}\cap(\widetilde{X}\smallsetminus\overline{K})\subseteq \mathfrak{c}(U)$. Then $\widetilde{E}\cap \left(\widetilde{X}\smallsetminus \overline{\widetilde{F}\cap(\widetilde{X}\smallsetminus\widetilde{E})}\right)\subseteq\mathfrak{c}(U)$, i.e., $\widetilde{E}\cap \left(\widetilde{X}\smallsetminus \overline{\widetilde{F}\cap\widetilde{X\smallsetminus E}}\right)\subseteq\mathfrak{c}(U)$. We get that \begin{align*}
			\widetilde{E}&\quad\subseteq\quad\mathfrak{c}(U)\vee \overline{\widetilde{F}\cap\widetilde{X\smallsetminus E}}\\
			&\quad=\quad\mathfrak{c}(U)\vee \mathfrak{c}\left(\bigwedge \left(\widetilde{F}\cap(\widetilde{X\smallsetminus E})\right)\right)\\
			&\quad=\quad\mathfrak{c}(U)\vee \mathfrak{c}(X\smallsetminus\overline{F\cap (X\smallsetminus E)})\quad\text{since }\bigwedge \left(\widetilde{F}\cap(\widetilde{X\smallsetminus E})\right)=X\smallsetminus\overline{F\cap(X\smallsetminus E)}\\
			&\quad=\quad\mathfrak{c}\left(U\cap (X\smallsetminus\overline{F\cap (X\smallsetminus E)})\right)\\
			&\quad=\quad S_{\left((X\smallsetminus U)\cup \overline{F\cap (X\smallsetminus E)}\right)}\quad\text{by \cite[Lemma 2.10.]{N}}.
		\end{align*}Therefore $E\subseteq (X\smallsetminus U)\cup \overline{F\cap (X\smallsetminus E)}$ so that $E\cap (X\smallsetminus\overline{F\cap (X\smallsetminus E)})\subseteq X\smallsetminus U$. Because $U$ is open, $X\smallsetminus U$ is closed, making $x\in \overline{E\cap (X\smallsetminus \overline{F\cap (X\smallsetminus E)})}\subseteq X\smallsetminus U$. Therefore $\widetilde{x}\in\widetilde{X\smallsetminus U}=\mathfrak{c}(U)$ implying that $\widetilde{x}\in \overline{\widetilde{E}\cap(\widetilde{X}\smallsetminus\overline{K})}$. Thus $$\widetilde{x}\in \overline{\widetilde{E}\cap(\widetilde{X}\smallsetminus\overline{K})}\cap\widetilde{E}=cl_{\widetilde{E}}\left(\widetilde{E}\cap(\widetilde{X}\smallsetminus\overline{K})\right)$$ which means that $\widetilde{x}$ is almost $\widetilde{E}$-inaccessible.

		$(\Longleftarrow):$ Let $N$ be an $(X\smallsetminus E)$-closed nowhere dense subset and set $N=F\cap (X\smallsetminus E)$ for some closed $F\subseteq X$. It follows from Proposition \ref{Knd} that $\widetilde{F}\cap\widetilde{X\smallsetminus E}=\widetilde{F}\cap(\widetilde{X}\smallsetminus\widetilde{E})$ is $(\widetilde{X\smallsetminus E}=\widetilde{X}\smallsetminus\widetilde{E})$-closed nowhere dense. Therefore $$\widetilde{x}\in \cl_{\widetilde{E}}\left(\widetilde{E}\cap \left(\widetilde{X}\smallsetminus\overline{\widetilde{F}\cap(\widetilde{X}\smallsetminus\widetilde{E})}\right)\right).$$ We show that $x\in \overline{E\cap (X\smallsetminus \overline{N})}$. Let $K$ be a closed set such that $E\cap (X\smallsetminus\overline{N})\subseteq K$. Then $E\subseteq \overline{N}\cup K$ so that $$\widetilde{E}\subseteq\widetilde{\overline{N}}\vee \widetilde{K}=\overline{\widetilde{N}}\vee\widetilde{K}=\overline{S_{(F\cap (X\smallsetminus E))}}\vee \widetilde{K}\subseteq \overline{\widetilde{F}\cap \left(\widetilde{X}\smallsetminus\widetilde{E}\right)}\vee\widetilde{K}.$$ Therefore $$\widetilde{E}\cap \left(\widetilde{X}\smallsetminus\overline{\widetilde{F}\cap\left(\widetilde{X}\smallsetminus\widetilde{E}\right)}\right)\subseteq\widetilde{K}.$$ Because $\widetilde{K}$ is a closed sublocale, $$\cl_{\widetilde{E}}\left(\widetilde{E}\cap \left(\widetilde{X}\smallsetminus\overline{\widetilde{F}\cap(\widetilde{X}\smallsetminus\widetilde{E})}\right)\right)=\widetilde{E}\cap \overline{\left(\widetilde{E}\cap \left(\widetilde{X}\smallsetminus\overline{\widetilde{F}\cap(\widetilde{X}\smallsetminus\widetilde{E})}\right)\right)}\subseteq \widetilde{K}$$ so that $\widetilde{x}\in\widetilde{K}$. Therefore $x\in K$. Thus $x\in \overline{E\cap (X\smallsetminus \overline{N})}$ which implies $x\in \cl_{E}\left(E\cap (X\smallsetminus \overline{N})\right)$. As a result, $x\notin\inte_{E}\left(E\cap \overline{N}\right)$. Hence $x$ is almost $E$-inaccessible.
	\end{proof}

	In terms of sublocales, we define  inaccessibility and almost inaccessibility on arbitrary locales. We give the following definition.
	
	\begin{definition}
		Let $S$ be a sublocale of $L$. A sublocale $T\in \mathcal{S}(S)$ is \textit{$S$-inaccessible} (resp. \textit{almost $S$-inaccessible}) if for all $(L\smallsetminus S)$-nowhere dense sublocale $N$, $T\cap \overline{N}=\mathsf{O}$ (resp. $T\subseteq\cl_{S}\left[S\cap \left(L\smallsetminus\overline{ N}\right)\right]$).
	\end{definition}
	
	We introduce the following notations for any locale $L$ and $S\in\mathcal{S}(L)$:
	\[\mathcal{S}_{\Inac}(S)=\{A\in \mathcal{S}(L):A\ \text{is}\ S \text{-inaccessible}\},\]and
	\[\mathcal{S}_{\Ainac}(S)=\{A\in \mathcal{S}(L):A\ \text{is}\ \text{almost}\ S \text{-inaccessible}\}.\]
	We shall drop prefix $S$- if the sublocale is clear from the context.
	
In what follows, we characterize inaccessible sublocales. The proof is similar to that of \cite[Proposition 3.10.]{N} and shall be omitted.
\begin{proposition}\label{inacccharact}
	The following are equivalent for a sublocale of $S$ of $L$. 
	\begin{enumerate}
		\item $T\in\mathcal{S}(S)$ is $S$-inaccessible.
		\item $T\cap \mathfrak{c}(x)=\mathsf{O}$ for each $(L\smallsetminus S)$-dense $x$.
		\item $T\subseteq \mathfrak{o}(x)$ for every $(L\smallsetminus S)$-dense $x$.
		\item $\nu_{T}(x)=1$ for each $(L\smallsetminus S)$-dense $x$.
	\end{enumerate} 
\end{proposition}
For sublocales $F$ and $A$ of a locale $L$, we have that $$F\subseteq\cl_{F}(F\cap A)\; \Longleftrightarrow \;F=\cl_{F}(F\cap A) \;\Longleftrightarrow F\cap A \text{ is } F\text{-dense}.$$ As a result of this, we have the following observation regarding almost inaccessible sublocales.

\begin{obs}\label{obsinacc}
	Let $F$ be a sublocale of a locale $L$. The following statements are equivalent.
	\begin{enumerate}
		\item $F\in\mathcal{S}_{\Ainac}(F)$.
		\item $F=\cl_{F}(F\cap (L\smallsetminus \overline{N}))$ for every $(L\smallsetminus F)$-nowhere dense $N$.
		\item $F\cap (L\smallsetminus\overline{N})$ is $F$-dense for every $(L\smallsetminus F)$-nowhere dense $N$.
	\end{enumerate}
\end{obs}

We give the following lemma which we shall use below.
\begin{lemma}\label{complalmost}
	Let $L$ be a locale, $N\in\mathcal{S}(L)$, $S$ a complemented sublocale of $L$ and $T\in\mathcal{S}(S)$. Then $T\subseteq \cl_{S}(S\cap \left(L\smallsetminus\overline{ N}\right))$ iff $T\cap \inte_{S}\left(S\cap \overline{N}\right)=\mathsf{O}.$
\end{lemma}\begin{proof}
	We have that
	\begin{align*}
		T\subseteq \cl_{S}(S\cap \left(L\smallsetminus\overline{ N}\right))
		&\quad\Longleftrightarrow\quad
		T\subseteq\overline{S\cap \left(L\smallsetminus\overline{N}\right)}\\
		&\quad\Longleftrightarrow\quad
		T\cap \left(L\smallsetminus\overline{S\cap (L\smallsetminus\overline{N})}\right)=\mathsf{O}\\
		&\quad\Longleftrightarrow\quad
		T\cap S\cap \left(L\smallsetminus\overline{S\cap (L\smallsetminus\overline{N})}\right)=\mathsf{O}\\
		&\quad\Longleftrightarrow\quad
		T\cap \inte_{S}\left(S\cap \overline{N}\right)=\mathsf{O}
	\end{align*}which proves the result.
\end{proof}

The preceding lemma leads us to the following characterization of almost inaccessible sublocales of complemented sublocales. We only prove the equivalences $(2)\Longleftrightarrow(3)$.

\begin{proposition}\label{complementedremote}
	The following are equivalent for a complemented sublocale $S$ of a locale $L$ and $T\in \mathcal{S}(S)$.
	\begin{enumerate}
		\item $T$ is almost $S$-inaccessible.
		\item $T\cap \inte_{S}(S\cap\overline{N})=\mathsf{O}$ for each $(L\smallsetminus S)$-nowhere dense sublocale $N$.
	\end{enumerate} 
	In particular, if $S$ is closed, this is further equivalent to:
	\begin{enumerate}
		\item[3.] $a\rightarrow \left(\bigwedge S\right)\leq\bigwedge T$ for every $(L\smallsetminus S)$-dense $a$.
	\end{enumerate}
\end{proposition}
\begin{proof}
	$(2)\Longleftrightarrow(3)$: Let $a$ be $(L\smallsetminus S)$-dense. Then $\mathfrak{c}_{(L\smallsetminus S)}(a)$ is $(L\smallsetminus S)$-nowhere dense. By $(2)$, \begin{align*}
		T\cap\inte_{S}(S\cap\overline{\mathfrak{c}_{(L\smallsetminus S)}(a)})=\mathsf{O}&\quad\Longleftrightarrow\quad T\subseteq \cl_{S}\left(S\cap \left(L\smallsetminus\overline{\mathfrak{c}_{(L\smallsetminus S)}(a)}\right)\right)\text{ since }S\text{ is complemented}\\
		&\quad\Longleftrightarrow\quad T\subseteq \overline{S\cap \left(L\smallsetminus\overline{\mathfrak{c}_{(L\smallsetminus S)}(a)}\right)}\\
		&\quad\Longleftrightarrow\quad T\subseteq \overline{S\cap \left(L\smallsetminus\mathfrak{c}(a)\right)}\quad\text{since }\overline{\mathfrak{c}_{(L\smallsetminus S)}(a)}=\mathfrak{c}(a)\\
		&\quad\Longleftrightarrow\quad T\subseteq \mathfrak{c}\left(\bigwedge\left(S\cap \mathfrak{o}(a)\right)\right)\\
		&\quad\Longleftrightarrow\quad T\subseteq \mathfrak{c}\left(a\rightarrow\left(\bigwedge S\right)\right)\quad\text{since }S\text{ is closed}\\
		&\quad\Longleftrightarrow\quad \overline{T}\subseteq \mathfrak{c}\left(a\rightarrow\left(\bigwedge S\right)\right)\\
		&\quad\Longleftrightarrow\quad a\rightarrow\left(\bigwedge S\right)\leq \bigwedge T.
	\end{align*}Starting the above argument with an $(L\smallsetminus S)$-nowhere dense $N$ and using the fact that $N$ is $(L\smallsetminus S)$-nowhere dense if and only if $\bigwedge N$ is $(L\smallsetminus S)$-dense, gives the desired equivalence.
\end{proof}

Next, we collect into one proposition some results about inaccessible sublocales and almost inaccessible sublocales.

\begin{proposition}\label{almostinaccesprop}
	Let $L$ be a locale and $S\in \mathcal{S}(L)$. 
	\begin{enumerate}
		\item Every $S$-inaccessible sublocale is almost $S$-inaccessible. 
		\item If $A\in\mathcal{S}(L)$ is $S$-inaccessible (resp. almost $S$-inaccessible) and $\mathcal{S}(L)\ni B\subseteq A$, then $B$ is $S$-inaccessible (resp. almost $S$-inaccessible). 
		\item If $S$ is open in $L$, then $S$ is $S$-inaccessible. 
		\item $L$ is $L$-inaccessible and hence by (2) every sublocale of $L$ is $L$-inaccessible. 
		\item If $S$ is complemented, then every sublocale $T$ of $L\smallsetminus S$ which is open in $L$ is $(L\smallsetminus S)$-inaccessible. 
		\item A join of $S$-inaccessible (resp. almost $S$-inaccessible) sublocales is $S$-inaccessible (resp. almost $S$-inaccessible). 
	\end{enumerate}
\end{proposition}
\begin{proof}
	(1) Let $T\in\mathcal{S}(S)$ be such that $T$ is $S$-inaccessible. Then $T\cap \overline{N}=\mathsf{O}$ for all $N\in\NdS(L\smallsetminus S)$, which implies that $T\subseteq L\smallsetminus\overline{N}$. Therefore $T\subseteq S\cap(L\smallsetminus\overline{N})\subseteq\cl_{S}(S\cap(L\smallsetminus\overline{N}))$ which proves the result.
	
	(2) Straightforward.
	
	(3) Assume that $S$ is open in $L$ and choose $N\in\NdS(L\smallsetminus S)$. It is clear that  $\overline{N}\subseteq L\smallsetminus S$ since $L\smallsetminus S$ is closed. Because $S$ is complemented, we have that $S\cap (L\smallsetminus S)=\mathsf{O}$ so that $S\cap \overline{N}=\mathsf{O}$. Thus $S\in \mathcal{S}_{\Inac}(S)$.
	
	(4) Because $L$ is open as a sublocale of itself, it follows from (3) that $L$ is $L$-inaccessible. Therefore, by (2), every sublocale of $L$ is $L$-inaccessible.
	
	(5) Let $T$ be a sublocale of $L\smallsetminus S$ which is open in $L$. We must show that $T\cap \overline{N}=\mathsf{O}$ for every $(L\smallsetminus(L\smallsetminus S))$-nowhere dense $N$, i.e., for every $S$-nowhere dense $N$. Since $T\subseteq L\smallsetminus S$, $T\cap S=\mathsf{O}$ because $S$ is complemented. So, for any $S$-nowhere dense $N$, $T\cap N=\mathsf{O}$. But $T$ is open in $L$ so $T\cap\overline{N}=\mathsf{O}$.
	
	(6) We only verify the case of $S$-inaccessible. Let $U_{i}\in\mathcal{S}_{\Inac}(S)$ (for $i\in I$) and choose $N\in\NdS(L\smallsetminus S)$. Since $\overline{N}$ is complemented, $\overline{N}\cap\bigvee U_{i}=\bigvee\left(\overline{N}\cap U_{i}\right)=\bigvee\{\mathsf{O}\}=\mathsf{O}$. Thus $\bigvee U_{i}\in \mathcal{S}_{\Inac}(S)$.
\end{proof}

\begin{remark}
	We note from Proposition \ref{almostinaccesprop}(4) that since every sublocale $S$ of a locale $L$ is a locale in its own right, it is therefore $S$-inaccessible as a sublocale of itself. However, in this paper, the notion $S\in \mathcal{S}_{\Inac}(S)$ for $S\in\mathcal{S}(L)$, read as $S$ is inaccessible as a sublocale of $L$ with respect to itself, shall mean $S\cap\overline{N}=\mathsf{O}$ for every $(L\smallsetminus S)$-nowhere dense sublocale $N$. This also applies to $S\in\mathcal{S}_{\Ainac}(S)$. 
\end{remark}

We note the following example.
\begin{example}\label{inacex}
	In a completely regular locale $L$, a point $p$ of $L$ is $\mathfrak{c}(p)$-inaccessible (resp. almost $\mathfrak{c}(p)$-inaccessible) if and only if $\mathfrak{c}(p)$ is $\mathfrak{c}(p)$-inaccessible (resp. almost $\mathfrak{c}(p)$-inaccessible).

\end{example}

	We give the following theorem in which some of its statements prepare us for a relationship between maximal nowhere density and remoteness in Proposition \ref{remoteandmaximal}. 
	
	\begin{theorem} \label{maximal}
		Let $L$ be a locale and $F$ be a non-void and closed nowhere dense sublocale of $L$. Then each of the following statements holds.
		\begin{enumerate}
			\item  If $F\in\mathcal{S}_{\Ainac}(F)$, then $F$ is maximal nowhere dense in $L$.
			\item If $L$ is compact, then $F$ is maximal nowhere dense in $L$ implies that there is $x\in F$ such that $x\notin\Int_{F}(\overline{ N}\cap F)$ for every $(L\smallsetminus F)$-nowhere dense $N$.
			\item $F$ is homogeneous maximal nowhere dense implies that every sublocale of $F$ is almost $F$-inaccessible. 
		\end{enumerate}
	\end{theorem}
	
	\begin{proof}
		(1) Let $F=\mathfrak{c}(b)$ for some $b\in L$ and choose $\mathfrak{c}(c)\in \NdS(L)$ such that $F\subseteq\mathfrak{c}(c)$. We show that $F$ is not nowhere dense in $\mathfrak{c}(c)$. Observe that $\mathfrak{c}(c)\smallsetminus\mathfrak{c}(b)\in\NdS(\mathfrak{o}(b))$. Indeed, if $$\mathfrak{o}(x)\cap\mathfrak{o}(b)\subseteq\overline{\mathfrak{c}(c)\smallsetminus\mathfrak{c}(b)}^{\mathfrak{o}(b)}=\overline{\mathfrak{c}(c)\cap\mathfrak{o}(b)}\cap\mathfrak{o}(b)=\mathfrak{c}(c)\cap\mathfrak{o}(b),$$ then $\mathfrak{o}(x)\cap\mathfrak{o}(b)=\mathfrak{o}(x\wedge b)\subseteq\mathfrak{c}(c)$. But $\mathfrak{c}(c)\in\NdS(L)$, so $\mathfrak{o}(x\wedge b)=\mathsf{O}$. Thus $\mathfrak{c}(c)\smallsetminus\mathfrak{c}(b)\in\NdS(\mathfrak{o}(b))$.
		
		Since $\mathfrak{c}(b)\in\mathcal{S}_{\Ainac}(\mathfrak{c}(b))$, we have that $$\mathfrak{c}(b)=\cl_{\mathfrak{c}(b)}\left(\mathfrak{c}(b)\cap\left(L\smallsetminus\overline{(\mathfrak{c}(c)\smallsetminus \mathfrak{c}(b))}\right)\right)=\cl_{\mathfrak{c}(b)}\left(\mathfrak{c}(b)\smallsetminus\overline{(\mathfrak{c}(c)\smallsetminus \mathfrak{c}(b))}\right).$$
		Because $\mathfrak{c}(b)$ is non-void, $\mathfrak{c}(b)\smallsetminus\overline{(\mathfrak{c}(c)\smallsetminus \mathfrak{c}(b))}\neq\mathsf{O}$. Since $\mathfrak{c}(b)\subseteq\mathfrak{c}(c)$,  $\mathfrak{c}(c)\smallsetminus\overline{(\mathfrak{c}(c)\smallsetminus \mathfrak{c}(b))}\neq\mathsf{O}$. 
		
		We must have that $\mathfrak{c}(b)\notin\NdS(\mathfrak{c}(c))$, otherwise 
		\begin{align*}
			\mathsf{O}&\quad=\quad\Int_{\mathfrak{c}(c)}\left(\overline{\mathfrak{c}(b)}^{\mathfrak{c}(c)}\right)\\
			&\quad =\quad
			\mathfrak{c}(c)\cap \left(L\smallsetminus \overline{(\mathfrak{c}(c)\cap L\smallsetminus \overline{\mathfrak{c}(b)})}\right)\\
			&\quad =\quad
			\mathfrak{c}(c)\cap \left(L\smallsetminus \overline{(\mathfrak{c}(c)\smallsetminus \mathfrak{c}(b))}\right)\\
			&\quad =\quad
			\mathfrak{c}(c)\smallsetminus\overline{(\mathfrak{c}(c)\smallsetminus \mathfrak{c}(b))}
		\end{align*}
		which is not possible. So there is no closed $K\in\NdS(L)$ such that $F\in\NdS(K)$. Thus $F$ is maximal nowhere dense in $L$.

		(2) 	Assume that $L$ is compact, $F=\mathfrak{c}(b)$ is maximal nowhere dense in $L$ and suppose that for each $x\in F$, there is an $\mathfrak{o}(b)$-nowhere dense $N_{x}$ such that $x\in\Int_{\mathfrak{c}(b)}(\mathfrak{c}(b)\cap\overline{ N_{x}})$. Set $\Int_{\mathfrak{c}(b)}(\mathfrak{c}(b)\cap\overline{ N_{x}})=\mathfrak{o}(a_{x})\cap \mathfrak{c}(b)$. Then $$\mathfrak{c}(b)\subseteq\bigvee_{x\in \mathfrak{c}(b)}(\mathfrak{o}(a_{x})\cap\mathfrak{c}(b))\subseteq\mathfrak{o}\left(\bigvee_{x\in \mathfrak{c}(b)}a_{x}\right).$$ Therefore $b\vee \left(\bigvee_{x\in \mathfrak{c}(b)}a_{x}\right)=1$. By compactness of $L$, there is a finite set $B\subseteq \mathfrak{c}(b)$ such that $b\vee\left(\bigvee_{x\in B}a_{x}\right)=1.$ We get that
		\begin{align*}
			\mathfrak{c}(b)
			&\quad \subseteq\quad
			\mathfrak{c}(b)\cap\mathfrak{o}\left(\bigvee_{x\in B}a_{x}\right)\\
			&\quad =\quad
			\mathfrak{c}(b)\cap \bigvee_{x\in B}\mathfrak{o}(a_{x})\\
			&\quad =\quad
			\bigvee_{x\in B}(\mathfrak{c}(b)\cap\mathfrak{o}(a_{x}))\\
			&\quad =\quad
			\bigvee_{x\in B}\left(\Int_{\mathfrak{c}(b)}(\mathfrak{c}(b)\cap\overline{ N_{x}})\right)\\
			&\quad \subseteq\quad
			\bigvee_{x\in B}\left(\mathfrak{c}(b)\cap\overline{ N_{x}}\right)\\
			&\quad =\quad
			\mathfrak{c}(b)\cap\bigvee_{x\in B}\overline{N_{x}}\\
			&\quad \subseteq\quad
			\bigvee_{x\in B}\overline{N_{x}}=\overline{ \bigvee_{x\in B}N_{x}}.
		\end{align*} 
		
		Observe that $N_{x}\in \NdS(L)$. This is so because $N_{x}$ is nowhere dense in a dense sublocale $\mathfrak{o}(b)$ of $L$ making it nowhere dense in $L$. Therefore $\overline{N_{x}}\in\NdS(L)$. Since finite joins of closed nowhere dense sublocales are nowhere dense, $\bigvee_{x\in B}\overline{N_{x}}=\overline{\bigvee_{x\in B}N_{x}}$ is nowhere dense in $L$.  We show that $\mathfrak{c}(b)$ is nowhere dense in $\overline{ \bigvee_{x\in B}N_{x}}$ which will contradict that $\mathfrak{c}(b)$ is maximal nowhere dense in $L$.

		Set $A=\overline{ \bigvee_{x\in B}N_{x}}$. Observe that
		\begin{align*}
			\Int_{A}\left(\overline{\mathfrak{c}(b)}^{A}\right)
			&\quad =\quad
			A\cap \left(L\smallsetminus \overline{(A\cap L\smallsetminus \overline{\mathfrak{c}(b)})}\right)\\
			&\quad =\quad
			A\cap \left(L\smallsetminus \overline{(A\cap \mathfrak{o}(b)}\right)\\
			&\quad =\quad
			A\cap \left(L\smallsetminus \overline{\left(\overline{\bigvee_{x\in B}N_{x}}\cap \mathfrak{o}(b)\right)}\right)\\
			&\quad \subseteq\quad
			A\cap \left(L\smallsetminus \overline{\left(\left(\bigvee_{x\in B}N_{x}\right)\cap \mathfrak{o}(b)\right)}\right)\\
			&\quad =\quad
			A\cap \left(L\smallsetminus \overline{\bigvee_{x\in B}N_{x}}\right)\quad \text{since}\quad \bigvee_{x\in B}N_{x}\subseteq \mathfrak{o}(b)\\
			&\quad =\quad
			\mathsf{O}.
		\end{align*} 
		Thus $\mathfrak{c}(b)$ is nowhere dense in $A$ which is a contradiction.

		(3) Suppose that there is $B\in\mathcal{S}(F)$ such that $B\notin\mathcal{S}_{\Ainac}(F)$. Then, by Proposition \ref{complementedremote}, $B\cap\Int_{F}(\overline{ N}\cap F)\neq\mathsf{O}$ for some $N\in \NdS(L\smallsetminus F)$. We get that $\Int_{F}(\overline{ N}\cap F)\neq\mathsf{O}$. Set $A=\Int_{F}(\overline{ N}\cap F)$. Then $\overline{A}^{F}=\overline{A}$ is a non-void $F$-regular-closed sublocale. Since $F$ is homogeneous maximal nowhere dense, $\overline{A}$ is maximal nowhere dense in $L$. We show that $\overline{A}$ is nowhere dense in $\overline{ N}$ which will contradict that it is a maximal nowhere dense sublocale. It is clear that $\overline{A}\subseteq \overline{ N}$. Furthermore, observe that 
		\begin{align*}
			\Int_{\overline{ N}}\left(\overline{A}^{\overline{N}}\right)
			&\quad =\quad
			\Int_{\overline{ N}}(\overline{A})\\
			&\quad =\quad
			\overline{ N} \cap \left(L\smallsetminus\overline{\overline{ N}\cap L\smallsetminus \overline{A}}\right)\\
			&\quad \subseteq\quad
			\overline{ N} \cap \left(L\smallsetminus\overline{\overline{ N}\cap L\smallsetminus F}\right)\quad \text{since}\quad\overline{A}\subseteq F\\
			&\quad \subseteq\quad
			\overline{ N} \cap \left(L\smallsetminus\overline{N\cap L\smallsetminus F}\right)\\
			&\quad =\quad
			\overline{ N} \cap \left(L\smallsetminus\overline{N}\right)\quad \text{since}\quad N\subseteq L\smallsetminus F\\
			&\quad =\quad
			\mathsf{O}.
		\end{align*}
		Thus $\overline{A}$ is nowhere dense in $\overline{ N}$ which is not possible. Hence every sublocale of $F$ belongs to $\mathcal{S}_{\Inac}(F)$.
	\end{proof}
	If we consider $F$-clopen sublocales, we get the following result.
	\begin{proposition}\label{hmndclopen}
		Let $F$ be a non-void nowhere dense sublocale of $L$. If $F$ is homogeneous maximal nowhere dense, then each $F$-clopen sublocale is almost inaccessible as a sublocale of $L$ with respect to itself.
	\end{proposition}
	\begin{proof}
		Let $A$ be an $F$-clopen sublocale and assume that $A\cap \Int_{A}(\overline{N}\cap A)\neq\mathsf{O}$ for some $N\in\NdS(L\smallsetminus A)$. Then $\Int_{A}(\overline{N}\cap A)$ is $F$-open so that $\overline{\Int_{A}(\overline{N}\cap A)}\cap F=\overline{\Int_{A}(\overline{N})}$ is a non-void regular-closed sublocale of $F$. By hypothesis, $\overline{\Int_{A}(\overline{N}\cap A)}$ is maximal nowhere dense. Following the argument used in last part of the proof of Theorem \ref{maximal}(3) and using the fact that $A$ is $F$-closed, we get that $\overline{\Int_{A}(\overline{N}\cap A)}$ is nowhere dense in $\overline{N}$ which is not possible.
	\end{proof}
	We note the following example.
	\begin{example}\label{exmnd}
		(1) If $X$ is a Hausdorff space with no isolated point, then every one-point sublocale of $\mathfrak{O}X$ which is almost inaccessible as a sublocale of $\mathfrak{O}X$ with respect to itself is maximal nowhere dense.
		To see this, it suffices to show that such sublocales are non-void closed nowhere dense in $\mathfrak{O}X$. Since every Hausdorff space is sober, each point $p$ of $\mathfrak{O}X$ is of the form $p=X\smallsetminus\overline{\{x\}}$ for some $x\in X$. Applying Hausdorffness again gives $p=X\smallsetminus\{x\}$ which is open and dense in $X$, making $\{x\}$ closed nowhere dense in $X$. It follows from \cite[Lemma 2.11.]{N} that $\widetilde{\{x\}}$ is closed and nowhere dense in $\mathfrak{O}X$. But $\widetilde{\{x\}}=\{X\smallsetminus\{x\},1_{\mathfrak{O}X}\}=\{p,1_{\mathfrak{O}X}\}$, so the one-point sublocale $\{p,1_{\mathfrak{O}X}\}$ is non-void closed nowhere dense. Now, if such a one-point sublocale $\{p,1_{\mathfrak{O}X}\}$ is an almost inaccessible sublocale of $L$ with respect to itself,  it follows from Proposition \ref{maximal}(1) that $\{p,1_{\mathfrak{O}X}\}$ is maximal nowhere dense.
		
		(2) The sublocales described in (1) are homogeneous maximal nowhere dense. This is so because for any point $p\in \mathfrak{O}X$, $\{p,1_{\mathfrak{O}X}\}$ is the only non-void sublocale contained in $\{p,1_{\mathfrak{O}X}\}$. Therefore all non-void $\{p,1_{\mathfrak{O}X}\}$-regular-closed sublocales are maximal nowhere dense in $\mathfrak{O}X$. 
	\end{example}

	We include the following result where we make use of Theorem \ref{maximal}(3) to show that a homogeneous maximal nowhere dense sublocale is regular-closed in every complemented nowhere dense sublocale containing it.
	
	\begin{proposition}
		Let $L$ be a locale and $F$ a non-void closed nowhere dense sublocale of $L$. If $F$ is homogeneous maximal nowhere dense and $F\subseteq A$, where $A$ is a complemented nowhere dense sublocale of $L$, then $F$ is an $A$-regular-closed sublocale.
	\end{proposition}
	\begin{proof}
		Assume that $F\neq \overline{\Int_{A}(F)}$. Because it is always true that $\overline{\inte_{A}(F)}\subseteq F$, this assumption says that $F\nsubseteq\overline{\inte_{A}(F)}$. Therefore the $F$-open sublocale $F\smallsetminus \overline{\Int_{A}(F)}=F\cap \left(L\smallsetminus\overline{\inte_{A}(F)}\right)$ is non-void. Also, $F\smallsetminus \overline{\Int_{A}(F)}\subseteq \overline{A\smallsetminus F}$. Indeed, 
		\begin{align*}
			F\smallsetminus\overline{\Int_{A}(F)}
			&\quad \subseteq\quad
			F\cap (L\smallsetminus \Int_{A}(F))\\
			&\quad =\quad
			F\cap \left(L\smallsetminus\left(A\cap (L\smallsetminus\overline{A\cap (L\smallsetminus F)})\right)\right)\\
			&\quad =\quad
			F\cap \left((L\smallsetminus A)\vee (L\smallsetminus (L\smallsetminus\overline{A\cap (L\smallsetminus F)}))\right)\\
			&\quad =\quad
			F\cap \left((L\smallsetminus A)\vee \overline{A\cap (L\smallsetminus F)}\right)\\
			&\quad =\quad
			(F\cap (L\smallsetminus A))\vee (F\cap\overline{A\cap (L\smallsetminus F)})\\
			&\quad =\quad
			\mathsf{O}\vee (F\cap\overline{A\cap (L\smallsetminus F)})\ \text{since}\ A\ \text{is complemented and }F\subseteq A\\
			&\quad =\quad
			F\cap\overline{A\cap (L\smallsetminus F)}\\
			&\quad\subseteq\quad\overline{A\cap (L\smallsetminus F)}=\overline{A\smallsetminus F}.
		\end{align*}
		This makes $\left(F\smallsetminus\overline{\Int_{A}(F)}\right)\cap \inte_{F}(F\cap\overline{A\smallsetminus F})\neq\mathsf{O}$.
		Observe that $A\smallsetminus F\in\NdS(L\smallsetminus F)$. To see this, let $U$ be an open sublocale of $L\smallsetminus F$ contained in $\overline{A\smallsetminus F}^{(L\smallsetminus F)}=\overline{A\smallsetminus F}\cap (L\smallsetminus F)$. Then $U\subseteq \overline{A}$. But $A\in \NdS(L)$ and an open sublocale of $L\smallsetminus F$ is open in $L$, so we have that $U=\mathsf{O}$. Thus $A\smallsetminus F\in\NdS(L\smallsetminus F)$. We have found a sublocale $F\smallsetminus\overline{\Int_{A}(F)}$ of $F$ and $A\smallsetminus F\in\NdS(L\smallsetminus F)$ such that $\left(F\smallsetminus\overline{\inte_{A}(F)}\right)\cap\Int_{F}\left(F\cap \overline{A\smallsetminus F}\right)\neq \mathsf{O}$, i.e, a sublocale of $F$ which is not almost $F$-inaccessible. By Theorem \ref{maximal}(3), $F$ is not homogeneous maximal nowhere dense, which is a contradiction.
	\end{proof}
	
In what follows, we characterize locales in which every non-void nowhere dense sublocale is maximal nowhere dense. Recall from \cite{D2} that the \textit{boundary} of a sublocale $S$ of a locale $L$ is given by $\bd(S)=\overline{S}\smallsetminus \inte(S)$.

\begin{theorem}\label{thmmnd}
	Let $L$ be a locale. The following statements are equivalent.
	\begin{enumerate}
		\item Every non-void nowhere dense sublocale of $L$ is maximal nowhere dense.
		\item Every non-void closed nowhere dense sublocale is maximal nowhere dense.
		\item Every non-void open sublocale induced by a non-complemented element of $L$ has a maximal nowhere dense boundary. 
		\item Every non-void closed nowhere dense sublocale of $L$ is homogeneous maximal nowhere dense.
		\item Every non-void closed nowhere dense sublocale is almost inaccessible as a sublocale of $L$ with respect to itself.
		\item Every non-void closed nowhere dense sublocale is inaccessible as a sublocale of $L$ with respect to itself.
	\end{enumerate}
\end{theorem}
\begin{proof}
	$(1)\Longleftrightarrow(2)$: Follows from Proposition \ref{mndcmnd}.
	
	$(2)\Longrightarrow(3)$: Let $\mathfrak{o}(x)\in\mathcal{S}(L)$ be non-void with $x$ non-complemented. Then $x\vee x^{*}\neq 1$ making $\bd(\mathfrak{o}(x))=\mathfrak{c}(x\vee x^{*})$ a non-void closed nowhere dense sublocale. It follows from (2) that $\bd(\mathfrak{o}(x))$ is maximal nowhere dense.
	
	$(3)\Longrightarrow(4)$: Let $\mathfrak{c}(x)$ be a non-void nowhere dense sublocale of $L$ and choose $y\in L$ such that $\overline{\mathfrak{o}(y)\cap\mathfrak{c}(x)}\cap\mathfrak{c}(x)\neq\mathsf{O}$. Then $$\overline{\mathfrak{o}(y)\cap\mathfrak{c}(x)}\cap\mathfrak{c}(x)=\overline{\mathfrak{o}(y)\cap\mathfrak{c}(x)}=\mathfrak{c}(y\rightarrow x)\neq\mathsf{O}$$ implying that $y\rightarrow x\neq 1$. But $\mathfrak{c}(y\rightarrow x)\subseteq\mathfrak{c}(x)\in\NdS(L)$, so $\mathfrak{c}(y\rightarrow x)\in\NdS(L)$ making $\mathfrak{o}(y\rightarrow x)$ non-void, open and $$(y\rightarrow x)\vee(y\rightarrow x)^{*}=(y\rightarrow x)\vee 0=y\rightarrow x\neq 1.$$ It follows from (3) that $$\bd(\mathfrak{o}(y\rightarrow x))=\mathfrak{c}((y\rightarrow x)\vee(y\rightarrow x)^{*})=\mathfrak{c}(y\rightarrow x)=\overline{\mathfrak{o}(y)\cap\mathfrak{c}(x)}$$ is maximal nowhere dense. Thus $\mathfrak{c}(x)$ is homogeneous maximal nowhere dense.

	$(4)\Longrightarrow(5)$: Follows from Theorem \ref{maximal}(3).
	
	$(5)\Longrightarrow(6)$: Let $\mathfrak{c}(x)$ be a non-void nowhere dense sublocale and assume that $\overline{N}\cap\mathfrak{c}(x)\neq\mathsf{O}$ for some $N\in\NdS((L\smallsetminus \mathfrak{c}(x))=\mathfrak{o}(x))$. Since $\mathfrak{o}(x)$ is dense in $L$, $N$ is nowhere dense in $L$ so that the non-void sublocale $\overline{N}\cap\mathfrak{c}(x)$ is closed nowhere dense in $L$. It follows from (5) that $\overline{N}\cap\mathfrak{c}(x)$ is almost $\left(\overline{N}\cap\mathfrak{c}(x)\right)$-inaccessible. Observe that $N\in \NdS(L\smallsetminus(\overline{N}\cap\mathfrak{c}(x)))$. To see this, let $a\in L$ be such that $\mathfrak{o}(a)\cap \left(L\smallsetminus(\overline{N}\cap\mathfrak{c}(x))\right)\subseteq\overline{N}\cap \left(L\smallsetminus(\overline{N}\cap\mathfrak{c}(x))\right)$. Then $\mathfrak{o}(a)\cap \left(L\smallsetminus(\overline{N}\cap\mathfrak{c}(x))\right)\subseteq\overline{N}$. Because $N\in\NdS(L)$ and $\mathfrak{o}(a)\cap \left(L\smallsetminus(\overline{N}\cap\mathfrak{c}(x))\right)$ is open, $\mathfrak{o}(a)\cap \left(L\smallsetminus(\overline{N}\cap\mathfrak{c}(x))\right)=\mathsf{O}$. Therefore $\mathfrak{o}(a)\subseteq \overline{N}\cap\mathfrak{c}(x)$ making $\mathfrak{o}(a)=\mathsf{O}$ since $\overline{N}\cap\mathfrak{c}(x)=\overline{\overline{N}\cap\mathfrak{c}(x)}$ is nowhere dense in $L$. Thus $N\in \NdS(L\smallsetminus(\overline{N}\cap\mathfrak{c}(x)))$. 
	
	Therefore $\overline{N}\cap \mathfrak{c}(x)\subseteq\overline{\left(\overline{N}\cap\mathfrak{c}(x)\right)\cap \left(L\smallsetminus\overline{ N}\right)}=\overline{\mathsf{O}}=\mathsf{O}$ which is not possible. Thus $\mathfrak{c}(x)\cap \overline{N}=\mathsf{O}$ making $\mathfrak{c}(x)\in\mathcal{S}_{\Inac}(\mathfrak{c}(x))$.
	
	$(6)\Longrightarrow(2)$: Let $F$ be a non-void closed nowhere dense sublocale of $L$. It follows from $(6)$ that $F\in\mathcal{S}_{\Inac}(F)$. Because $\mathcal{S}_{\Inac}(F)\subseteq\mathcal{S}_{\Ainac}(F)$ by Proposition \ref{almostinaccesprop}(1), it follows from Theorem \ref{maximal}(1) that $F$ is maximal nowhere dense. 
\end{proof}
	
	In the following example we show that there is a locale having the properties described in Theorem \ref{thmmnd}.
	
	\begin{example}
		Consider the three-element chain $\three=\{1,0,a\}$. Clearly, $\three$ is non-Boolean and the only non-void closed nowhere dense sublocale of $\three$ is $\mathfrak{c}(a)$ which is maximal nowhere dense because it is not nowhere dense as a sublocale of itself. 
	\end{example}

	\section{Maximal nowhere density and Remoteness}\label{remotenessandmnd}
	In this section, we consider a connection between remoteness and maximal nowhere density. Inaccessibility will be useful here.
	
	We recall the following notations from \cite{N,N1}. Let $L$ be a locale, $S$ a dense sublocale of $L$ and $T\in \mathcal{S}(L)$. Then
	\begin{enumerate}
		\item $T$ is a \textit{remote sublocale} of $L$ in case $T\cap N=\mathsf{O}$ for every nowhere dense $N\in\mathcal{S}(L)$.
		\item $\mathcal{S}_{\text{rem}}(L)$ denotes the collection of all remote sublocales of $L$.
		\item $T$ is \textit{remote from} $S$, denoted $T\in \mathcal{S}_{\text{rem}}(L\ltimes S)$, if $T\cap\overline{N}=\mathsf{O}$ for each $S$-nowhere dense $N$.
		\item $T$ is \textit{$^{*}\!$remote from} $S$, denoted $T\in ^{*}\!\mathcal{S}_{\text{rem}}(L\ltimes S)$, provided that $T\in \mathcal{S}_{\text{rem}}(L\ltimes S)$ and $T\subseteq L\smallsetminus S$. 
	\end{enumerate}
\begin{obs}
	A simple case where inaccessibility differs from remoteness is that of Proposition \ref{almostinaccesprop}(4). Recall from \cite{N} that a necessary and sufficient condition for a locale $L$ to be remote as a sublocale of itself is that it must be Boolean. Yet, by Proposition \ref{almostinaccesprop}(4), every locale $L$ (not necessarily Boolean) is $L$-inaccessible as a sublocale of itself.
\end{obs}
Observe that for each dense and complemented $S\in \mathcal{S}(L)$, $T\in \mathcal{S}(L\smallsetminus S)$ is $(L\smallsetminus S)$-inaccessible if and only if $T\cap \overline{N}=\mathsf{O}$ for all $(L\smallsetminus (L\smallsetminus S))$-nowhere dense sublocale $N$, i.e., for all $S$-nowhere dense sublocale $N$. This shows that sublocales of $L\smallsetminus S$ which are $ ^{*} $remote from a dense and complemented sublocale $S$ of $L$ are precisely the $(L\smallsetminus S)$-inaccessible sublocales. We formalise this in the following proposition.
\begin{proposition}\label{inaccessibilityand remoteness}
	A sublocale $T\in \mathcal{S}(L\smallsetminus S)$ where $S$ is dense and complemented in a locale $L$, is $(L\smallsetminus S)$-inaccessible iff it is $^{*}\!$remote from $S$.
\end{proposition}

In the following result, we codify the variants of remoteness and inaccessibility for dense and complemented sublocales.

\begin{proposition}\label{inaccrem}
	Let $L$ be a locale, $S$ a dense and complemented sublocale of $L$ and $T\in\mathcal{S}(L\smallsetminus S)$. Consider the following statements:
	\begin{enumerate}
		\item $T\in\mathcal{S}_{\text{rem}}(L)$.
		\item $T\in\mathcal{S}_{\text{rem}}(L\ltimes S)$.
		\item $T\in{^{*}\!\mathcal{S}_{\text{rem}}(L\ltimes S)}$.
		\item $T\in\mathcal{S}_{\Inac}(L\smallsetminus S)$.
		\item $T\in\mathcal{S}_{\Ainac}(L\smallsetminus S)$.
	\end{enumerate} Then $(1)\Longrightarrow(2)\Longleftrightarrow(3)\Longleftrightarrow(4)\Longrightarrow(5)$.
\end{proposition}
\begin{proof}
	$(1)\Longrightarrow(2)$: Follows from \cite[Proposition 3.5.(4).]{N}
	
	$(2)\Longleftrightarrow(3)$: This is a combination of \cite[Proposition 3.2.]{N} and the fact that $T\subseteq L\smallsetminus S$.
	
	$(3)\Longleftrightarrow(4)$: Follows from Proposition \ref{inaccessibilityand remoteness}.

	$(4)\Longrightarrow(5)$: Follows from Proposition \ref{almostinaccesprop}(1).
\end{proof}
\begin{example}
		\cite{V1} In \textbf{Top}, we have that a point $p\in \beta X\smallsetminus X$ is remote if and only if $p$ is $(\beta X\smallsetminus X)$-inaccessible.
\end{example}

We also include the following result which shows inaccessibility of the least dense sublocale.
\begin{proposition}\label{almostinaccesprop1}
	Let $L$ be a locale and $S\in \mathcal{S}(L)$.  If $S$ is open, then $S\cap \Rs(L\ltimes S)$ is $S$-inaccessible, where $\Rs(L\ltimes S)=\bigvee\{A\in\mathcal{S}(L):A\text{ is remote from } S\}$. Moreover, if $S$ is dense and open, then $\mathfrak{B}L$ is $S$-inaccessible.
\end{proposition}
\begin{proof}
	The first part follows from Proposition \ref{almostinaccesprop}(3) and (2). 

Since, according to Observation \cite[Observation 3.1.]{N}, $\mathfrak{B}L=T\cap\Rs(L\ltimes T)$ for any dense $T\in\mathcal{S}(L)$, we have that $\mathfrak{B}L=S\cap\Rs(L\ltimes S)$ is $S$-inaccessible for dense and open $S$. 
\end{proof}
	Using Proposition \ref{inaccrem} and the fact if $S\in\mathcal{S}(L)$ is open and dense, then $S^{\#}=L\smallsetminus S$ is nowhere dense, we get the following result about remoteness and maximal nowhere density.

\begin{proposition}\label{remoteandmaximal}
	Let $S\neq L$ be an open dense sublocale of $L$.
	\begin{enumerate}
		\item If $S^{\#}\in{^{*}\!\mathcal{S}_{\text{rem}}(L\ltimes S )}$, then $S^{\#}$ is maximal nowhere dense in $L$.
		\item If $S^{\#}$ is homogeneous maximal nowhere dense, then every $S^{\#}$-remote sublocale is $^{*}$remote from $S$.
	\end{enumerate}
\end{proposition}
\begin{proof}
	(1) If $S^{\#}\in{^{*}\!\mathcal{S}_{\text{rem}}(L\ltimes S )}$, then, by Proposition \ref{inaccrem}, $S^{\#}\in\mathcal{S}_{\Ainac}(S^{\#})$. It follows from Theorem \ref{maximal}(1) that $S^{\#}$ is maximal nowhere dense in $L$.

	(2) Let $A\in\mathcal{S}_{\text{rem}}(S^{\#})$ and choose an $S$-nowhere dense $N$. Since, by Theorem \ref{maximal}(3), sublocales of homogeneous maximal nowhere dense sublocales are almost inaccessible as sublocales of $L$ with respect to themselves, $S^{\#}$ is almost $S^{\#}$-inaccessible, i.e., $\inte_{S^{\#}}(S^{\#}\cap\overline{N})=\mathsf{O}$. Since $S^{\#}\cap\overline{N}=\overline{S^{\#}\cap\overline{N}}^{S^{\#}}$, we get that $S^{\#}\cap\overline{N}$ is $S^{\#}$-nowhere dense. Because $A\in\mathcal{S}_{\text{rem}}(S^{\#})$, $$\mathsf{O}=A\cap S^{\#}\cap\overline{N}=A\cap\overline{N}.$$ Thus $A\in\mathcal{S}_{\text{rem}}(L\ltimes S)$ making $A\in {^{*}\!\mathcal{S}_{\text{rem}}(L\ltimes S)}$ since $A\subseteq S^{\#}$.
\end{proof}
	\section{Preservation of maximal nowhere density}
	
	We end this paper with a discussion of  localic maps that send maximal nowhere density back and forth. We shall also include results about inaccessibility.
	
	\begin{proposition}\label{mapmnd}
		Let $f:L\rightarrow M$ be a localic map such that both $f$ and $f^{*}$ send dense elements to dense elements. Then $f_{-1}$ preserves maximal nowhere dense sublocales.
	\end{proposition}
	\begin{proof}
		Let $N\in\mathcal{S}(M)$ be maximal nowhere dense in $M$. Then $\overline{N}$ is maximal nowhere dense in $L$. 
		Since $f^{*}$ is weakly open, it follows from \cite[theorem 2.29]{N} that $f_{-1}$ preserves closed nowhere dense sublocales so that $f_{-1}[\overline{N}]$ (which is equal to $\mathfrak{c}(f^{*}(\bigwedge N))$) is nowhere dense in $L$. It is left to show that it is maximal nowhere dense. Suppose not, that is, there exists a nowhere dense sublocale $\mathfrak{c}(b)$ of $L$ such that $f_{-1}[\overline{N}]$ is nowhere dense in $\mathfrak{c}(b)$. Since $f$ sends dense elements to dense elements, $f(b)$ is dense in $M$ making $\mathfrak{c}(f(b))$ nowhere dense in $M$. The sublocale $\overline{N}$ is nowhere dense in the nowhere dense sublocale $\mathfrak{c}(f(b))$. Indeed, if $\mathfrak{o}(x)\cap \mathfrak{c}(f(b))\subseteq \overline{N}^{\mathfrak{c}(f(b))}=\overline{N}\cap\mathfrak{c}(f(b))$ for some $x\in M$, then $\mathfrak{o}(x)\cap \mathfrak{c}(f(b))\subseteq \overline{N}$ so that $$\mathfrak{o}(f^{*}(x))\cap \mathfrak{c}(b)\subseteq \mathfrak{o}(f^{*}(x))\cap\mathfrak{c}(f^{*}(f(b)))=f_{-1}[\mathfrak{o}(x)\cap \mathfrak{c}(f(b))]\subseteq  f_{-1}[\overline{N}].$$ Therefore $\mathfrak{o}(f^{*}(x))\cap \mathfrak{c}(b)\subseteq f_{-1}[\overline{N}]\cap \mathfrak{c}(b)$. Since $f_{-1}[\overline{N}]\in\NdS(\mathfrak{c}(b))$ and $\mathfrak{o}(f^{*}(x))\cap\mathfrak{c}(b)$ is open in $\mathfrak{c}(b)$, $\mathfrak{o}(f^{*}(x))\cap \mathfrak{c}(b)=\mathsf{O}$ which implies that $\mathfrak{c}(b)\subseteq \mathfrak{c}(f^{*}(x))=f_{-1}[\mathfrak{c}(x)]$. Therefore $f[\mathfrak{c}(b)]\subseteq f[f_{-1}[\mathfrak{c}(x)]]\subseteq\mathfrak{c}(x)$ implying that $\mathfrak{c}(f(b))=\overline{f[\mathfrak{c}(b)]}\subseteq \overline{\mathfrak{c}(x)}=\mathfrak{c}(x)$. This makes $\mathfrak{c}(f(b))\cap\mathfrak{o}(x)=\mathsf{O}$. Therefore $\overline{N}$ is nowhere dense in $\mathfrak{c}(f(b))$ which contradicts that $\overline{N}$ is maximal nowhere dense in $M$. Therefore $f_{-1}[N]$ is maximal nowhere dense in $L$.
	\end{proof}

	In the next result, we discuss localic maps that preserve maximal nowhere dense sublocales. 
	
	Recall that a frame homomorphism $h:M\rightarrow L$ is \textit{open} if it has a left adjoint $h_{!}$  satisfying the Frobenius identity: $$h_{!}(h(a)\wedge b)=h_{!}(a)\wedge b$$ for every $a\in M$ and $b\in L$. This is equivalent to saying that $h_{*}$ is an open localic map. Because, by \cite{BP2}, every open frame homomorphism is weakly open, it follows that for each open localic map $f$, $f^{*}$ is weakly open. 
	\begin{obs}\label{obsopen}
		Not every open localic map sends dense elements to dense elements. Consider the localic map $f:L\rightarrow\two$ where $L$ is non-Boolean and $\two$ is the two-element locale. Since $\two$ is Boolean, every sublocale of $\two$ is open making the localic image of each open sublocale of $L$ to be open in $\two$. Hence $f$ is open. However, $f$ does not send all dense elements to dense elements since the only (dense) element of $L$ that is mapped to $1_{\two}$ (the only dense element of $\two$) is $1$. But $1$ is not the only dense element of $L$ otherwise $L$ is Boolean.
	\end{obs}
	We recall from \cite{PPT} that if a localic map $f:L\rightarrow M$ is open, then $f_{-1}[\overline{A}]=\overline{f_{-1}[A]}$ for each $A\in\mathcal{S}(M)$. 
	\begin{proposition}\label{smndmap}
		Let $f:L\rightarrow M$ be an open localic map that sends dense elements to dense elements. Then $f$ preserves maximal nowhere dense sublocales.
	\end{proposition} 
	\begin{proof}
		Let $N$ be a maximal nowhere dense sublocale of $L$. We show that $\overline{f[N]}$ is maximal nowhere dense in $M$. By \cite[Lemma 2.33]{N}, $f[N]$ is nowhere dense in $M$ so that $\overline{f[N]}=\mathfrak{c}\left(f(\bigwedge N)\right)$ is nowhere dense in $M$. Suppose that $\mathfrak{c}\left(f(\bigwedge N)\right)\in\NdS(\mathfrak{c}(y))$ for some $\mathfrak{c}(y)\in\NdS(M)$. By Observation \ref{obsndsubspace}, $\mathfrak{c}(y)\subseteq\overline{\mathfrak{c}(y)\cap\mathfrak{o}\left(f(\bigwedge N)\right)}$. Therefore $f_{-1}[\mathfrak{c}(y)]=\mathfrak{c}(f^{*}(y))\subseteq f_{-1}\left[\overline{\mathfrak{c}(y)\cap\mathfrak{o}\left(f(\bigwedge N)\right)}\right]$. Openness of $f$ gives $f_{-1}\left[\overline{\mathfrak{c}(y)\cap\mathfrak{o}\left(f(\bigwedge N)\right)}\right]=\overline{f_{-1}[\mathfrak{c}(y)\cap\mathfrak{o}\left(f(\bigwedge N)\right)]}$, so that
		\begin{align*}
			\mathfrak{c}(f^{*}(y))&\quad\subseteq\quad\overline{f_{-1}\left[\mathfrak{c}(y)\cap \mathfrak{o}\left(f(\bigwedge N)\right)\right]}\\
			&\quad=\quad
			\overline{\mathfrak{c}(f^{*}(y))\cap \mathfrak{o}\left(f^{*}\left(f\left(\bigwedge N\right)\right)\right)}\\
			&\quad\subseteq\quad
			\overline{\mathfrak{c}(f^{*}(y))\cap \mathfrak{o}\left(\bigwedge N\right)}\\
			&\quad=\quad
			\overline{\mathfrak{c}(f^{*}(y))\cap \left(L\smallsetminus\overline{N}\right)}.
		\end{align*} Therefore $\mathfrak{c}(f^{*}(y))\cap \left(L\smallsetminus\overline{\mathfrak{c}(f^{*}(y))\cap \left(L\smallsetminus\overline{N}\right)}\right)=\mathsf{O}$. This makes $N\in\NdS(\mathfrak{c}(f^{*}(y)))$, where $\mathfrak{c}(f^{*}(y))\in\NdS(L)$, contradicting that $N$ is maximal nowhere dense in $L$. Therefore $\overline{f[N]}$ is maximal nowhere dense so that by Proposition \ref{mndcmnd}, $f[N]$ is maximal nowhere dense sublocales.
	\end{proof}
	Since in Proposition \ref{mapmnd} we only needed a condition that both $f$ and $f^{*}$ send dense elements to dense elements and because the left adjoint of an open localic map is weakly open, we have the following result.
	\begin{corollary}
		Every open localic map that sends dense elements to dense elements preserves and reflects maximal nowhere dense sublocales.
	\end{corollary}
	
	In the next result, we discuss preservation and reflection of strongly  homogeneous maximal nowhere dense sublocales by localic maps. 
	
	\begin{proposition}\label{hmndpresereflec}
		Let $f:L\rightarrow M$ be an open localic map that sends dense elements to dense elements. 
		\begin{enumerate}
			\item Then $f$ preserves weakly homogeneous maximal nowhere dense sublocales.
			\item If $f$ is injective, then it reflects (strongly) homogeneous maximal nowhere dense sublocales.
		\end{enumerate}
	\end{proposition}
	\begin{proof}
		(1) Let $F$ be a weakly homogeneous maximal nowhere dense sublocale of $L$ and choose a non-void sublocale $\overline{\mathfrak{o}(y)\cap f[F]}\cap f[F]$ where $y\in M$. Such a sublocale is $f[F]$-regular-closed. The $F$-regular-closed sublocale $\overline{\mathfrak{o}(f^{*}(y))\cap F}\cap F$ is non-void otherwise, $\mathfrak{o}(f^{*}(y))\cap F=\mathsf{O}$ so that $f[F]\subseteq f[\mathfrak{c}(f^{*}(y))]=f[f_{-1}[\mathfrak{c}(y)]]\subseteq\mathfrak{c}(y)$. Therefore $f[F]\cap \mathfrak{o}(y)=\mathsf{O}$ which is not possible. Since $F$ is strongly h.m.n.d, $\overline{\mathfrak{o}(f^{*}(y))\cap F}\cap F$ is m.n.d. Because open localic maps that send dense elements to dense elements preserve maximal nowhere dense sublocales (by Proposition \ref{smndmap}), $f\left[\overline{\mathfrak{o}(f^{*}(y))\cap F}\cap F\right]$ is m.n.d in $M$. Since \begin{align*}
			f\left[\overline{\mathfrak{o}(f^{*}(y))\cap F}\cap F\right]&\quad\subseteq\quad f\left[\overline{\mathfrak{o}(f^{*}(y))\cap F}\right]\cap f[F]\\
			&\quad\subseteq\quad \overline{f[\mathfrak{o}(f^{*}(y))\cap F]}\cap f[F]\\
			&\quad\subseteq\quad \overline{f[\mathfrak{o}(f^{*}(y))]\cap f[F]}\cap f[F]\\
			&\quad\subseteq\quad\overline{\mathfrak{o}(y)\cap f[F]}\cap f[F]
		\end{align*} and because $\overline{\mathfrak{o}(y)\cap f[F]}\cap f[F]$ is nowhere dense in $M$, it follows from Proposition \ref{mndprop}(2) that $\overline{\mathfrak{o}(y)\cap f[F]}\cap f[F]$ is m.n.d. Thus $f[F]$ is weakly homogeneous maximal nowhere dense in $M$.

		(2) We only prove reflection of weakly homogeneous maximal nowhere dense sublocales. That of homogeneous maximal nowhere dense sublocales follows the same sketch. Let $K$ be a weakly homogeneous maximal nowhere dense sublocale of $M$ and consider a non-void sublocale $\overline{\mathfrak{o}(x)\cap f_{-1}[K]}\cap f_{-1}[K]$ where $x\in L$. We must show that this $f_{-1}[K]$-regular-closed sublocale is m.n.d. We have that $\overline{\mathfrak{o}(f(x))\cap K}\cap K$ is a non-void $K$-regular-closed sublocale. To see that is it non-void, observe that having $\overline{\mathfrak{o}(f(x))\cap K}\cap K=\mathsf{O}$ implies that $\mathfrak{o}(f(x))\cap K=\mathsf{O}$ so that $$\mathsf{O}=f_{-1}[\mathfrak{o}(f(x))]\cap f_{-1}[K]=\mathfrak{o}(f^{*}(f(x)))\cap f_{-1}[K]=\mathfrak{o}(x)\cap f_{-1}[K]$$ where the latter equality follows from injectivity of $f$. This cannot be true, so $\overline{\mathfrak{o}(f(x))\cap K}\cap K$ is non-void. Since $K$ is weakly homogeneous maximal nowhere dense, $\overline{\mathfrak{o}(f(x))\cap K}\cap K$ is m.n.d in $M$. Because open localic maps are weakly open and $f$ sends dense elements to dense elements, it follows from Proposition \ref{mapmnd} that $f_{-1}[\overline{\mathfrak{o}(f(x))\cap K}\cap K]$ is m.n.d. Observe that $$f_{-1}[\overline{\mathfrak{o}(f(x))\cap K}\cap K]=f_{-1}[\overline{\mathfrak{o}(f(x))\cap K}]\cap f_{-1}[K]=\overline{f_{-1}[\mathfrak{o}(f(x))]\cap f_{-1}[K]}\cap f_{-1}[K]$$ where the latter equality follows from openness of $f$. By injectivity of $f$, $$f_{-1}[\overline{\mathfrak{o}(f(x))\cap K}\cap K]=\overline{\mathfrak{o}(x)\cap f_{-1}[K]}\cap f_{-1}[K]$$making $\overline{\mathfrak{o}(x)\cap f_{-1}[K]}\cap f_{-1}[K]$ m.n.d. in $L$. Thus $f_{-1}[K]$ is weakly homogeneous maximal nowhere dense in $L$.
	\end{proof}
	\begin{obs}
		For the preservation of homogeneous maximal nowhere dense sublocales, the localic map $f$ in Proposition \ref{hmndpresereflec} must also preserve closed sublocales. That is, it must also be closed which is a rather too stringent condition.
	\end{obs}
	Open localic maps also allow us to study, under certain conditions, preservation and reflection of inaccessible and almost inaccessible sublocales as presented below. 
	\begin{proposition}\label{inaccmap}
		Let $f:L\rightarrow M$ be an open and injective localic map. Then for all open $S\in\mathcal{S}(L)$,  \begin{enumerate}
			\item $f[\mathcal{S}_{\Inac}(S)]\subseteq \mathcal{S}_{\Inac}(f[S])$, and
			\item $f[\mathcal{S}_{\Ainac}(S)]\subseteq \mathcal{S}_{\Ainac}(f[S])$.
		\end{enumerate}
	\end{proposition} 
	\begin{proof}
		Follows from Proposition \ref{almostinaccesprop}(3), (2) and (1).
	\end{proof}
	\begin{proposition}\label{inacclocalicmap}
		Let $f:L\rightarrow M$ be localic map such that both $f$ and $f^{*}$ send dense elements to dense elements and let $T\in\mathcal{S}(M)$ be closed nowhere dense. Then  \begin{enumerate}
			\item $f_{-1}[\mathcal{S}_{\Inac}(T)]\subseteq\mathcal{S}_{\Inac}(f_{-1}[T])$, and 
			\item If $f$ is open, then $f_{-1}[\mathcal{S}_{\Ainac}(T)]\subseteq\mathcal{S}_{\Inac}(f_{-1}[T])$.
		\end{enumerate} 
	\end{proposition}
	\begin{proof} (1) Let $T\in\mathcal{S}(M)$ be closed and nowhere dense and choose $A\in \mathcal{S}_{\Inac}(T)$. Then $f_{-1}[A]\subseteq f_{-1}[T]$. Let $N\in\NdS(L\smallsetminus f_{-1}[T])$. Since, by \cite[(5.1)]{FPP}, $P\smallsetminus g_{-1}[C]\subseteq g_{-1}[R\smallsetminus C]$ for every localic map $g:P\rightarrow R$ with $C\in\mathcal{S}(R)$, we get that $N\subseteq L\smallsetminus f_{-1}[T]\subseteq f_{-1}[M\smallsetminus T]$. Therefore $f[N]\subseteq f[f_{-1}[M\smallsetminus T]]\subseteq M\smallsetminus T$. We show that $f[N]\in\NdS(M\smallsetminus T)$. Because $f^{*}$ is weakly open, it follows from \cite[Theorem 2.29.]{N} that $f_{-1}[-]:\mathcal{S}(M)\rightarrow\mathcal{S}(L)$ preserves closed nowhere dense sublocales so that $f_{-1}[T]$ is closed nowhere dense in $L$. Therefore $L\smallsetminus f_{-1}[T]$ is open and dense. Now, $N$ being nowhere dense in $L\smallsetminus f_{-1}[T]$ implies $N$ is nowhere dense in $L$. Since $f$ sends dense elements to dense elements, it follows from \cite[Lemma 2.33.]{N} that $f[-]$ preserves nowhere dense sublocales so that $f[N]$ is nowhere dense in $M$. Observe that $f[N]$ is nowhere dense in $M\smallsetminus T$. To see this, let $y\in M$ be such that $\mathfrak{o}(y)\cap (M\smallsetminus T)\subseteq \overline{f[N]}\cap (M\smallsetminus T)$. Then $\mathfrak{o}(y)\cap(M\smallsetminus T)\subseteq\overline{f[N]}$. Because $M\smallsetminus T$ is open in $M$ and $f[N]\in\NdS(M)$, we have that $\mathfrak{o}(y)\cap (M\smallsetminus T)=\mathsf{O}$ making $\mathfrak{o}(y)\subseteq T$. But $T$ is nowhere dense in $M$, so $\mathfrak{o}(y)=\mathsf{O}$ implying that \begin{equation}\label{eqf1}
			f[N]\in\NdS(M\smallsetminus T).
		\end{equation} $S$-inaccessibility of $A$ implies $A\cap \overline{f[N]}=\mathsf{O}$. Therefore 
		$$\mathsf{O}=f_{-1}\left[A\cap \overline{f[N]}\right]=f_{-1}[A]\cap f_{-1}\left[\overline{f[N]}\right]=f_{-1}[A]\cap \mathfrak{c}\left(h\left(f\left(\bigwedge N\right)\right)\right)\supseteq f_{-1}[A]\cap \overline{N}.$$ Thus $f_{-1}[A]\in\mathcal{S}_{\Inac}(f_{-1}[T])$.

		(2) Assume that $f$ is open. Set $T=\mathfrak{c}(b)$ for some $b\in M$ and choose $A\in\mathcal{S}_{\Ainac}(T)$. (\ref{eqf1}) still holds, so
		\[A\subseteq \mathfrak{c}(b)\cap \overline{\mathfrak{c}(b)\cap \left(M\smallsetminus \overline{f[N]}\right)}\\
		= \mathfrak{c}(b)\cap \overline{\mathfrak{c}(b)\cap \mathfrak{o}\left(\bigwedge f[N]\right)}.\] Therefore 
		\begin{align*}
			f_{-1}[A]&\quad\subseteq\quad  \mathfrak{c}(f^{*}(b))\cap f_{-1}\left[\overline{\mathfrak{c}(b)\cap \mathfrak{o}\left(\bigwedge f[N]\right)}\right]\\
			&\quad=\quad \mathfrak{c}(f^{*}(b))\cap \overline{f_{-1}\left[\mathfrak{c}(b)\cap \mathfrak{o}\left(\bigwedge f[N]\right)\right]}\quad\text{since }f \text{ is open}\\
			&\quad=\quad \mathfrak{c}(f^{*}(b))\cap \overline{\mathfrak{c}(f^{*}(b))\cap \mathfrak{o}\left(h\left(f\left(\bigwedge N\right)\right)\right)}\\
			&\quad\subseteq\quad \mathfrak{c}(f^{*}(b))\cap \overline{\mathfrak{c}(f^{*}(b))\cap \mathfrak{o}\left(\bigwedge N\right)}\\
			&\quad=\quad f_{-1}[T]\cap \overline{f_{-1}[T]\cap \left(L\smallsetminus\overline{N}\right)}\\
			&\quad=\quad \cl_{f_{-1}[T]}\left(f_{-1}[T]\cap \left(L\smallsetminus\overline{N}\right)\right).
		\end{align*}Thus $f_{-1}[A]\in \mathcal{S}_{\Ainac}(f_{-1}[T])$.
	\end{proof}
	


\end{document}